\documentclass[final,onetabnum,onefignum,onethmnum]{siamart190516}

\usepackage{amsfonts}
\usepackage{dsfont}
\usepackage{enumitem}

\newtheorem{example}{Example}
\title{A Unification of Weighted and Unweighted Particle Filters\thanks{Submitted to the editors November 26, 2020; last revision March 1, 2022.
\funding{This research was funded by the Swiss National Science Foundation (SNSF) grants PP00P3\_179060 and 31003A\_175644.}}}

\author{
    Ehsan Abedi\thanks{University of Bern, Department of Physiology, 3012, Bern, Switzerland (\href{mailto:ehsan.abedi@alumni.epfl.ch}{ehsan.abedi@alumni.\allowbreak epfl.ch}, \email{simone.surace@unibe.ch}}).
    \and 
    Simone Carlo Surace\footnotemark[2]
    \and
    Jean-Pascal Pfister\footnotemark[2] \thanks{UZH/ETH Zurich, Institute of Neuroinformatics, 8057, Zurich, Switzerland, and University of Bern, Department of Physiology, 3012, Bern, Switzerland (\email{jeanpascal.pfister@unibe.ch}).}
    }

\headers{Unification of Particle Filters}{E. Abedi, S.C. Surace, and J.-P. Pfister}

\begin{document}
\maketitle

\begin{abstract}
    Particle filters (PFs), which are successful methods for approximating the solution of the filtering problem, can be divided into two types: weighted and unweighted PFs.
    It is well known that weighted PFs suffer from the weight degeneracy and curse of dimensionality.
    To sidestep these issues, unweighted PFs have been gaining attention, though they have their own challenges.
    The existing literature on these types of PFs is based on distinct approaches.
    In order to establish a connection, we put forward a framework that unifies weighted and unweighted PFs in the continuous-time filtering problem.
    We show that the stochastic dynamics of a particle system described by a pair process, representing particles and their importance weights, should satisfy two necessary conditions in order for its distribution to match the solution of the Kushner--Stratonovich equation.
    In particular, we demonstrate that the bootstrap particle filter (BPF), which relies on importance sampling, and the feedback particle filter (FPF), which is an unweighted PF based on optimal control, arise as special cases from a broad class and that there is a smooth transition between the two.
    The freedom in designing the PF dynamics opens up potential ways to address the existing issues in the aforementioned algorithms, namely weight degeneracy in the BPF and gain estimation in the FPF.
\end{abstract}



\section{Introduction}\label{sec:introduction}
\subsection{Filtering problem}\label{subsec:problem}
    The goal of filtering is to dynamically estimate a latent variable from noisy observations.
    We consider the continuous-time nonlinear filtering problem (in one dimension for the sake of simplicity), where $X_t \in \mathbb{R}$ is the \emph{hidden} process satisfying an Itô stochastic differential equation (SDE) and $Y_t \in \mathbb{R}$ is the \emph{observation} process evolving according to an Itô SDE which depends on $X_t$:
    \begin{align}
        dX_t &=  f(X_t,t)\,dt + g(X_t,t)\, dB_t^{\text{\tiny X}}, \quad X_0 \sim P_0, \label{eq:hidden}\\
        dY_t &=  h(X_t,t)\,dt +  \,dB_t^{\text{\tiny Y}}, \label{eq:observation}
    \end{align}
    where $B_t^{\text{\tiny X}} , B_t^{\text{\tiny Y}} \in \mathbb{R}$ are independent Brownian motions (BMs) and $f(x,t)$, $g(x,t)$, and $h(x,t)$ are (known) functions that map $\mathbb{R}\times\mathbb{R}_{\geq 0}\mapsto\mathbb{R}$ and are called the drift, diffusion, and observation function, respectively.
    The initial condition $X_0$ is independent of the BMs and has (known) distribution $P_0$ with finite second moment and $Y_0=0$ (no observation at first).
    Let $\mathcal{B}(\mathbb{R})$ be the Borel $\sigma$-algebra on $\mathbb{R}$.
    The filtering problem is to find the (regular) conditional distribution of the hidden process $X_t$ given the history of observations $\mathcal{F}^{Y}_t := \sigma ({Y}_s : 0 \leq s \leq t)$, that is, $P_t(B) := \mathrm{Pr}(X_t \in B|\mathcal{F}^{Y}_t) $ for any Borel subset $B \in \mathcal{B}(\mathbb{R})$. 
    This distribution is referred to as the \emph{filtering distribution} and in the statistics literature it is also called the \emph{posterior distribution}, while $\mathrm{Pr}(X_t\in B)$ is called the \emph{prior distribution}. 
    We let $p(x,t)$ denote the probability density function corresponding to $P_t$ with respect to the Lebesgue measure, if it exists.
    
    Throughout the paper we assume that 
    $f,g$ satisfy the conditions for the well-posedness of SDEs (e.g., locally Lipschitz in $x$ uniformly in $t$; see, e.g., \cite[Theorem 5.4]{Klebaner2005}).
    In addition, to ensure the existence of the filtering density $p$, we either assume that $g^2\geq\delta>0$ is bounded below by a constant and that $\smash{f \in C^{1,0}_b}$, $\smash{g^2 \in C^{2,0}_b}$, and $\smash{h \in C^{1,0}_b}$, or we consider the linear-Gaussian case.
    Here $\smash{C^{\kappa,j}_b}$ denotes the space of bounded continuous functions with bounded partial derivatives up to order $\kappa$ in $x$ and $j$ in $t$ (see \cite[Theorem 7.11 and 7.17]{Bain2009} for details on existence, uniqueness, and smoothness of $p$).
    
    The stochastic processes in this paper are defined on a filtered probability space $(\Omega,\mathcal{F},\mathbb{P},(\mathcal{F}_t)_{t\geq 0})$ (satisfying the usual conditions) and are assumed to be progressively measurable (hence adapted) with respect to the filtration 
    $(\mathcal{F}_t)_{t\geq 0}$.
    We further denote by $\mathbb{L}^\kappa(0,T)$ the space of processes $(F_t)_{T\geq t\geq 0}$ with
    $\smash{\mathbb{E} [\int_{0}^{T}|F_t|^\kappa dt] < \infty}.$
    All SDEs are in the Itô sense, and primes ($'$) denote the (partial) derivative with respect to $x$.
    
\subsection{Formal solution}\label{subsec:solution}
    It is well known that the evolution of $p$ is described by the Kushner--Stratonovich equation (KSE) \cite{Kushner1964,Stratonovich1960}
	\begin{equation}\label{eq:KSE_dis}
		dp(x,t) =
		\mathcal{L}_t^{\dagger}p(x,t)\,dt + p(x,t) \big(h(x,t) - \hat{h}_t\big) (dY_t- \hat{h}_tdt),
	\end{equation}
    where $\smash{\mathcal{L}_t^{\dagger}\cdot = - \frac{\partial}{\partial x}(f(x,t)\cdot) + 
    \frac{1}{2} \frac{\partial^2}{\partial x^2}(g^2(x,t) \cdot)}$ is the adjoint Fokker--Planck operator and $\hat{h}_t : = \mathbb{E}[h(X_t,t) | \mathcal{F}^{Y}_t] = \int_{\mathbb{R}} h(x,t) p(x,t) dx $.
    The initial condition is assumed to be $p(x,0)=p_0(x) \in C^2$.
    The KSE consists of two parts: the first part is associated with the prior dynamics given by the Fokker--Planck equation, while the second part can be interpreted as a correction resulting from observations, which is proportional to the so-called \emph{innovation} term $(dY_t- \hat{h}_tdt)$. Equation \cref{eq:KSE_dis} can be equivalently converted into the evolution of a given statistic using integration by parts. Let $\phi\in C^2$ such that $\mathbb{E}[|\phi(X_t)|]<\infty$ for all $t\geq 0$. The conditional expectation of $\phi(X_t)$ evolves as 
    \begin{equation}\label{eq:KSE_mom}
    	d\mathbb{E}[\phi(X_t) | \mathcal{F}^{Y}_t] =
    	\mathbb{E} [\mathcal{L}_t\phi(X_t) | \mathcal{F}^{Y}_t]\,dt +
    	\mathbb{E} [ \phi(X_t) (h(X_t,t) - \hat{h}_t) | \mathcal{F}^{Y}_t ] (dY_t-\hat{h}_t dt),
    \end{equation}
    where $\smash{ \mathcal{L}_t \cdot = f(x,t) \frac{\partial}{\partial x} \cdot + \frac{1}{2}g^2(x,t) \frac{\partial^2}{\partial x^2} \cdot}$ is the generator of the process $X_t$.
    
    \begin{example}[linear-Gaussian case]
    We will use this simple case throughout the paper as an illustration of key concepts.
    The linear-Gaussian case is characterized by linear drift terms and additive noise, as well as a Gaussian initial distribution:
    \begin{equation}\label{eq:linear_problem}
        f(x,t) = ax, \quad g(x,t) = b, \quad h(x,t) = c x, \qquad P_0 = \mathcal{N}(\mu_0,\rho_0),
    \end{equation}
    where $a,b,c \in \mathbb{R}$.
    In this case, the KSE \cref{eq:KSE_dis} can be solved in closed form.
    We have $ X_t|\mathcal{F}^Y_t \sim \mathcal{N}(\hat{\mu}_t,\hat{\rho}_t )$ given by the Kalman-Bucy \cite{Kalman1961} classical result 
	\begin{align}
	d\hat{\mu}_t & =  a\hat{\mu}_t\,dt + c\hat{\rho}_t (dY_t - c\hat{\mu}_t\,dt ),    \label{eq:posterior_mean} \\
	d\hat{\rho}_t & = b^2\,dt + 2a\hat{\rho}_t\,dt -c^2 \hat{\rho}_t^2\,dt.  \label{eq:posterior_var}
	\end{align}
    The coefficient of the innovation term in \cref{eq:posterior_mean} is called the Kalman gain $\bar{K}(t):= c\hat{\rho}_t$.
    \end{example}
    
    In contrast to the linear-Gaussian case, for most signal and observation models \cref{eq:KSE_dis} does not have closed-form solutions. 
    Likewise, \cref{eq:KSE_mom}, when it is applied to the moments of the filtering distribution, gives rise to a closure problem in which the evolution of the $n$th moment $\mathbb{E}[X_t^n | \mathcal{F}^{Y}_t]$ generally depends on higher-order moments.
    Therefore, the KSE is only a “formal” solution to the filtering problem and needs to be approximated numerically in practice. 
    Among the numerical methods, particle filters (PFs) have been widely and successfully applied because of their versatility.
    Below, they will be presented within a broader framework, but see e.g., \cite[section 8.6]{Bain2009}, , and the tutorials \cite{Doucet2009,Kutschireiter2020} for more detailed surveys of the PFs.

\subsection{Particle filters}\label{subsec:PFs}
    These methods are aimed at approximating the filtering distribution by the empirical distribution of a particle system, which in full generality is a triangular array of random variables \cite{Chopin2004} for some fixed $t$,
    \begin{equation}\label{eq:particle_system}
        \{S_t^{(i,N)}, w_t^{(i,N)}\}_{i=1}^{N} \quad \text{with}\,  \sum\nolimits_{i=1}^{N} w_t^{(i,N)} = 1,
    \end{equation}
    where $S_t^{(i,N)}$ are samples, also called \emph{particles}, and $w_t^{(i,N)}$ are their corresponding \emph{importance weights}, which without loss of generality are assumed to be normalized. 
    To justify such a method, one has to study the $N$-particle system and show that the sequence of empirical distributions converges (at least in a weak sense) to the filtering distribution as $N \to \infty$. 
    In this article, however, we use a mean-field-limit approach; i.e., we study abstract “particle systems” characterized by a \emph{pair process} denoted by $(S_t,W_t)$. 
    The following definition is modified from \cite[section 2.1]{Chopin2004}.
    \begin{definition}[targeting condition]\label{def:targets}
        The particle system described by a pair process $(S_t,W_t)$, representing particles and their weights respectively, is said to target the filtering distribution in the filtering problem \cref{eq:hidden}--\cref{eq:observation} at time $t$ if and only if
        \begin{equation}\label{eq:targets}
        	\mathbb{E}[W_t \phi(S_t) | \mathcal{F}^{Y}_t] = \mathbb{E}[\phi(X_t) | \mathcal{F}^{Y}_t] \quad \text{a.s.}
        \end{equation}
        holds for any measurable function $\phi(x):\mathbb{R}\mapsto\mathbb{R}$ such that $\mathbb{E}[|\phi(X_t)|]<\infty$.
    \end{definition}
    
    It should be noted that weights for a given $S_t$ are \emph{not unique}. For instance, if the pair $(S_t,W_t)$ satisfies the targeting condition, then any pair $(S_t,W_t+V_t)$ is also a solution, where $V_t$ is a stochastic process independent of $S_t$ with zero conditional mean and finite second moment. In \cref{subsec:weights}, we further clarify this nonuniqueness. 
    
    To have the targeting condition over a period of time, i.e., a dynamic version of \cref{eq:targets}, we are interested in the time-evolution of $(S_t,W_t)$.  
    As we shall see later in \cref{subsec:unifying}, this naturally leads to \emph{McKean--Vlasov} SDEs, in which the coefficients in the dynamics of $(S_t,W_t)$ become dependent on the targeted distribution. 
    In numerical implementations, the left-hand side of \cref{eq:targets} is approximated by a Monte Carlo estimate $\smash{\frac{1}{M} \sum_{i=1}^{N} W_t^{(i,N)}\phi(S_t^{(i,N)})}$ using samples $(S_t^{(i,N)},W_t^{(i,N)})$, $i=1,\ldots,N$, from the joint distribution of $(S_t,W_t)$ given $\mathcal{F}^{Y}_t$, where $M :=\sum_{i=1}^N W_t^{(i,N)}$, which is of order $N$.
    If the samples are chosen appropriately (e.g., i.i.d.), the Monte Carlo estimate converges to the conditional expectation by a law of large numbers and yields asymptotic consistency of the PF.
    The normalized weights are then given by $\smash{w_t^{(i,N)}=\tfrac{1}{M} W_t^{(i,N)}}$. 
    Thus, in theory, if we combine \cref{def:targets} with appropriate sampling, we obtain an asymptotically exact filter.
    In practice, however, propagation of the samples requires the McKean--Vlasov terms to be estimated based on the current sample.
    This estimation problem is nontrivial and introduces correlations between samples, which complicates the convergence analysis.
    Although interesting and worthwhile, the issues of estimation and convergence are not within the scope of the present article.
    Instead, we focus on the characterization of abstract particle systems.
    
    A particle filter is said to be \emph{unweighted} if $W_t=1$ for all $t$; otherwise, it is called \emph{weighted}. 
    Observe that in particular, \cref{eq:targets} implies (by setting $\phi=1$) that $\mathbb{E}[W_t|\mathcal{F}^{Y}_t]=1$ a.s. and due to the nonnegativity of the variance of $W_t$, we also have $\mathbb{E}[W_t^2|\mathcal{F}_t^Y] \geq 1$ a.s. 
    If $W_t$ deviates significantly from unity, this means that the Monte Carlo variance of the weighted average is larger than it would be for an average with equal weights.
    This can be measured by the \emph{effective sample size}, a commonly used approximation of which (see, e.g., \cite[section 2]{Martino2017}) is given by
        \begin{equation}\label{eq:Neff}
	        N_{\text{eff}} : =  \Big( \sum\nolimits_{i=1}^{N}  {(w_t^{(i,N)}})^2 \Big)^{-1}.
	    \end{equation}
	 Since $\smash{w_t^{(i,N)}=\frac{1}{M} W_t^{(i,N)}}$ as explained before, by a law of large numbers we have $M/N\to 1$ as $N\to\infty$ and therefore
	    \begin{equation}
	        \frac{N}{N_{\text{eff}}}=\Big(\frac{N}{M}\Big)^2\frac{1}{N}\sum\nolimits_{i=1}^N \Big(W_t^{(i,N)}\Big)^2\, \xrightarrow{p} \, \mathbb{E}[W_t^2|\mathcal{F}_t^Y] \qquad \text{as } N \to \infty,
	    \end{equation}
	provided samples become i.i.d. for large $N$.   
	So to keep $N_{\text{eff}}$ close to $N$, it is desirable that $\text{Var}[W_t|\mathcal{F}_t^Y]=\mathbb{E}[W_t^2|\mathcal{F}_t^Y]-1$ remain close to zero. 
	In order to assess the degeneracy of a PF algorithm, we are essentially interested in minimizing the unconditional variance of the importance weights, $\text{Var}[W_t]=\mathbb{E}[W_t^2]-1$, taking all possible realizations of the observation process into account \cite{Doucet2000}. In \cref{subsec:weights}, we will establish the minimum-variance weight for a fixed particle distribution.
	We now review two well-known examples of PFs within the framework above.
    
    $\bullet$ \textbf{The bootstrap particle filter (BPF)}.
    The BPF is a weighted PF that was originally introduced by \cite{Gordon1993} and is widely used in discrete-time filtering \cite{Doucet2009}.
    Here we present its continuous-time formulation (see, e.g., \cite[chapter 9]{Bain2009} or \cite[section 6.1]{Kutschireiter2020}).
    The particles in this filter move with the same law as the hidden process, thereby being distributed according to the prior.
    The weight dynamics must consequently include observations in such a way that the weighted particles are distributed according to the posterior.
    The evolution of the particle system denoted by $(S^{\text{\tiny B}}_t,W^{\text{\tiny B}}_t)$ reads as
    \begin{align}
        dS^{\text{\tiny B}}_t & = f (S^{\text{\tiny B}}_t,t) \, dt + g (S^{\text{\tiny B}}_t,t) \, dB_t, \label{eq:BPF_dS} \\ 
        dW^{\text{\tiny B}}_t & = W^{\text{\tiny B}}_t (h(S^{\text{\tiny B}}_t,t) - \hat{h}_t) \, (dY_t-\hat{h}_tdt), \label{eq:BPF_dW}
    \end{align}
    with initial condition $S^{\text{\tiny B}}_0 \sim P_0$ and $W^{\text{\tiny B}}_0=1$, and $B_t$ being a BM independent of \{$(B^{\text{\tiny X}}_t)_{t\geq 0}$, $(B^{\text{\tiny Y}}_t)_{t\geq 0}$, $X_0$, $S^{\text{\tiny B}}_0$\} (note that in practice, particles are driven by independent BMs). 
    The usual derivation of the BPF is based on a change of probability measure, in which $W_t^{\text{\tiny B}}$ is the evaluation of the Radon–Nikodym derivative $d\mathbb{P}/d\mathbb{Q}$ with $X_t$ replaced by $S_t^{\text{\tiny B}}$, where $\mathbb{P}$ is the (original) coupled measure of the system $(X_t,Y_t)$ and $\mathbb{Q}$ is a new measure under which $X_t$ and $Y_t$ are independent, the dynamics of $X_t$ remains unchanged, and $Y_t$ corresponds to a BM, $dY_t=dB_t^{\text{\tiny Y}}$ \cite[chapter 9]{Bain2009}:
    \begin{equation} \label{eq:BPF_W}
        W^{\text{\tiny B}}_t=\frac{\tilde W^{\text{\tiny B}}_t}{\mathbb{E}_{\mathbb{Q}}[\tilde W^{\text{\tiny B}}_t|\mathcal{F}^Y_t]}, \quad \tilde W^{\text{\tiny B}}_t:=\mathbb{E}_{\mathbb{Q}}\bigg[\frac{d\mathbb{P}}{d\mathbb{Q}}\bigg|\mathcal{F}_t\bigg].
    \end{equation}
    Girsanov's theorem and It\^o's formula then give rise to \cref{eq:BPF_dW}.
    
    Implementing the dynamics \cref{eq:BPF_dS}--\cref{eq:BPF_dW} in practice is straightforward. 
    However, the BPF suffers from the weight decay; that is, most of the weights become negligibly small and only a few of them remain significant, an issue which becomes even more severe in high dimensions \cite{Surace2019}. 
    A common practice to overcome this issue is to periodically resample the particles, in which case \cref{eq:BPF_dS}--\cref{eq:BPF_dW} merely describe the evolution of the particle system between the resampling times. 
    Although resampling techniques are key ingredients of weighted PFs, they are not the focus of the current work. Instead, we attempt to find ways to alleviate the weight collapse itself. \vspace{1pt}
   
     $\bullet$ \textbf{The feedback particle filter (FPF).}
     The FPF is an unweighted PF that was introduced in \cite{FPF2013}, initially motivated by mean-field optimal control.
     In contrast to the BPF, this filter does not have weight dynamics. 
     Instead, the particles must incorporate the observations and \emph{interact} with each other so that they can target the filtering distribution by themselves.
     The key idea is to add a correction term, also called the control input, to the prior dynamics of particles,
     \begin{align}
        dS^{\text{\tiny F}}_t & = f (S^{\text{\tiny F}}_t,t)\,dt + g (S^{\text{\tiny F}}_t,t)\,dB_t + U(S^{\text{\tiny F}}_t,t)\,dt + K(S^{\text{\tiny F}}_t,t)\,dY_t, \label{eq:FPF_dS} \\
        dW^{\text{\tiny F}}_t & = 0, \label{eq:FPF_dW}
    \end{align}
    and find the unknown functions $U$ and $K$ such that the conditional density of $S_t$ given $\mathcal{F}^{Y}_t$ solves the KSE \cref{eq:KSE_dis}. 
    Here again the initial condition is $S^{\text{\tiny F}}_0 \sim P_0$ and $W^{\text{\tiny F}}_0=1$ and $B_t$ is a BM independent of \{$(B^{\text{\tiny X}}_t)_{t\geq 0}$, $(B^{\text{\tiny Y}}_t)_{t\geq 0}$, $X_0$, $S^{\text{\tiny F}}_0$\}.
    The paper \cite{FPF2013} shows that $U(x,t)$ and $K(x,t)$ under certain technical assumptions must satisfy
    \begin{gather}
        \frac{\partial}{\partial x} ( K p) + ( h - \hat{h}_t ) p = 0, \label{eq:FPF_K}\\
        U =  - \frac{1}{2} K \big( h + \hat{h}_t - \frac{\partial}{\partial x} K \big) \label{eq:FPF_U}.
    \end{gather}
    The main challenge is to find the so-called \emph{gain function} $K$, which in turn depends on $p$. 
    In the multidimensional case, \eqref{eq:FPF_K} does not have uniqueness of solutions because any solution $K$ can generate another solution by adding a divergence-free vector field.
    A commonly used solution is obtained (uniquely) by restricting $K$ to be of gradient form \cite{FPF2016}, where \eqref{eq:FPF_K} then becomes a weighted \textit{Poisson equation}.
    Different assumptions on $K$ is one aspect of how different PFs arise.
    Although we are aware of similar filters, as noted in \cite{Pathiraja2020}, that require solving equations like \eqref{eq:FPF_K} and have been referred to as “particle flow filters,” we refer to algorithm above as the FPF.
    
    \begin{example}[linear-Gaussian case, continued]
    In this case, \eqref{eq:FPF_K}--\eqref{eq:FPF_U} have closed form solutions and the gradient form solution for $K$ equals the Kalman gain.
    \end{example}
    
    While fixing $K$ to be in gradient form is useful to pick a gain in practice and makes the boundary value problem accessible to a range of numerical approximations such as the RKHS method \cite{Radhakrishnan2018}, its necessity is not well justified from a theoretical perspective. 
    This lack of justification is especially striking on smooth manifolds without a Riemannian metric given a priori, where the gradient field depends on the chosen metric \cite{Abedi2019, Surace2020}. 
    One final point to notice is that the FPF does not require any resampling procedure, which is an advantage compared to the BPF. 
    Moreover, \cite{Surace2019} shows numerically that the FPF is less prone to the curse of dimensionality (COD) than the BPF. 
    \cref{table:BPF_vs_FPF} summarizes the comparison between the BPF and FPF. 
    
\begin{table}
    {\footnotesize
    \begin{center}
    \label{table:BPF_vs_FPF}
    \caption{A comparison between the BPF and the FPF as they have been treated in the literature so far.
    Each has its strengths $(+)$ and weaknesses $(-)$.}
	\begin{tabular}{lll}
		\hline
		Bootstrap particle filter \cite{Gordon1993}
		&
		$\quad$
		&
		Feedback particle filter \cite{FPF2013} \\
		\hline
		a weighted particle filter && an unweighted particle filter\\
		based on change of probability measure && motivated by optimal control\\
		asymptotically exact $(+)$ && asymptotically exact $(+)$\\
		suffers from weight degeneracy $(-)$ && requires gain estimation $(-)$\\
		exhibits the COD $(-)$ && potentially avoids the COD $(+)$\\
		\hline
	\end{tabular}
	\end{center}}
\end{table}

\subsection{Motivation and contribution}
    So far we have seen two well-known PFs, each of which was originally based on a distinct approach, yet they both satisfy the targeting condition (\cref{def:targets}).
    Given the comparison between these methods, it is still unknown whether combining the strengths of both types is possible.
    As a step towards this goal, we provide a unified treatment of a large family of weighted and unweighted PFs that is based solely on the targeting condition.
    Specifically, we do the following:
    \begin{itemize}
        \item We first characterize the weights $W_t$ in terms of a Radon--Nikodym derivative of the marginal distributions over $S_t$ and find the optimal (i.e., minimum-variance) weight $W_t^*$ for fixed particle distribution (see \cref{thm:weights}). This result also sheds some light on the nonuniqueness of importance weights in particle filtering.
    	\item We then introduce a general dynamics for the particle system $(S_t,W_t)$ expressed as a system of SDEs and obtain necessary conditions on its coefficients for targeting the filtering distribution (see \cref{thm:unifying}, referred to as the “unifying theorem”). 
    	The necessary conditions are a system of ordinary differential equations (ODEs) in one dimension, which becomes a system of partial differential equations (PDEs) in higher dimensions.
    	The results also hold in the unconditional case (see \cref{crl:unifying_unconditional}).
    	\item As a result of the unifying theorem, we derive a class of PFs which encompasses the BPF and FPF with a smooth transition between them, thereby indicating that these methods are not different in their nature (see \cref{prop:class}).
    	\item The optimal importance weight $W_t^*$ from the first theorem is studied in the context of the unifying theorem, and its evolution $dW_t^*$ is derived (see \cref{prop:dW_optimal}).
    	\item Finally, we outline two potential applications of the unifying theorem in \cref{sec:applications}, namely compensating for gain estimation errors with weight dynamics and providing freedom to alleviate the weight degeneracy.
    \end{itemize}
    For the sake of simplicity, we develop the results in the one-dimensional setting.
    In principle, the results can be generalized to the multidimensional case; however, further consideration is required, as more freedom emerges in higher dimensions.
    
\subsection{Related work}
    A unifying framework for discrete- and continuous-time filtering (as a Bayesian formulation of the data assimilation problem) from the perspective of couplings, optimal transport, and Schr\"odinger bridges is proposed in \cite{Reich2019}.
    Similarly, in \cite{Pathiraja2020}, three types of unweighted PFs, including the FPF, are unified in the framework of McKean--Vlasov SDEs.
    However, to the best of our knowledge a unification of weighted and unweighted approaches such as this one has not been attempted. 
    This question is also loosely connected to the concept of proposal distributions in discrete-time particle filters.
    From the perspective of proposal distributions in the sequential Monte Carlo (SMC) sampling literature, the BPF can be viewed as the filter for which the proposal is equal to the prior transition density of the hidden state. 
    Nevertheless, it is unclear how to view an optimal proposal distribution as in \cite[section D]{Doucet2000} in relation to continuous-time filters, in particular the FPF.
    It was pointed out in \cite[section 3.2]{Surace2019} that the optimal proposal becomes trivial in the continuous-time limit.
    In the same paper, also the broader question of how to reconcile the dynamics of the FPF with the importance weights as Radon--Nikodym derivatives of path measures was posed. 
    The present paper answers this question by adopting a change-of-measure approach that is a generalization of the path-measure framework adopted in the literature on continuous-time PFs.
    
    This paper also touches on the notion of nonuniqueness in designing PFs, which has been discussed in the literature but mainly restricted to the linear-Gaussian case and freedom of the particle movements.
    For example, \cite{Abedi2019} provides a systematic exploration of the nonuniqueness within the class of linear FPF in terms of gauge transforms and \cite{Taghvaei2020} studies the nonuniqueness of the feedback control law in particle dynamics for different types of ensemble Kalman filters.
    The aforementioned paper \cite{Pathiraja2020} also explores the nonuniqueness in unweighted PFs and gives a general formulation from which existing filters can be obtained as special cases by making specific assumptions on the form of the coefficients. 
    In this work, we examine the nonuniqueness in a more general setting that includes weight dynamics.

\section{Main results}\label{sec:results}
\subsection{A characterization of weights in terms of the Radon--Nikodym derivatives}\label{subsec:weights}
    The approach used to derive the importance weights in the BPF \cref{eq:BPF_W} is difficult to reconcile with filters such as the FPF, in which observation terms are included in the particle dynamics.
    This explicit $dY$ term makes the measure of the particles singular with respect to the measure of the hidden process (see \cite[section 3.2]{Surace2019} for a discussion of this issue).
    Here we demonstrate a general relationship between the process $W_t$ and the Radon--Nikodym derivative of the marginal distributions over $S_t$, which is consistent with both the BPF and FPF, and we also provide the minimum-variance choice for $W_t$ (for a fixed particle distribution).
    
    \begin{theorem}\label{thm:weights}
     Consider the filtering problem \cref{eq:hidden}--\cref{eq:observation} with filtering distribution $P_t(B) := \mathrm{Pr}(X_t \in B|\mathcal{F}^{Y}_t) $ for any $B \in \mathcal{B}(\mathbb{R})$. 
     Let $(S_t, W_t)$ be a pair process characterizing an abstract particle system. 
     Denote by $Q_t$ the conditional distribution of $S_t$ given $\mathcal{F}^Y_t$, i.e., $Q_t(B) := \mathrm{Pr}(S_t \in B|\mathcal{F}^{Y}_t) $. 
     At any time $t$, under the condition that the particle system described by $(S_t,W_t)$ targets $P_t$, according to \cref{def:targets}, we have the following:
    \begin{enumerate}[label=\Roman*.,leftmargin=*]
        \item\label{thm:weights_i}The distribution $P_t$ is absolutely continuous with respect to the distribution $Q_t$.
        In particular, the Radon--Nikodym derivative $dP_t/dQ_t$ exists, and
        \begin{equation}\label{eq:weight_Radon-Nikodym}
        \mathbb{E}[W_t |\mathcal{F}^{Y}_t, S_t] = \frac{dP_t}{dQ_t}(S_t) \quad a.s.
        \end{equation}
        \item\label{thm:weights_ii}$W^*_t := \frac{dP_t}{dQ_t}(S_t)$ mimimizes $\mathbb{E}[W_t^2]$ subject to \cref{eq:weight_Radon-Nikodym}.
    \end{enumerate}
    \end{theorem}

    The proof is given in \cref{prf:weights}, and by reversing the arguments in the proof it is obvious that \eqref{eq:weight_Radon-Nikodym} is already sufficient to guarantee the targeting condition \cref{eq:targets}.
    Notice also that if $P_t$ and $Q_t$ have densities $p(x,t)$ and $q(x,t)$ with respect to the Lebesgue measure, then \cref{eq:weight_Radon-Nikodym} turns into
    \begin{equation}\label{eq:weight_Radon-Nikodym_density}
        \mathbb{E}[W_t |\mathcal{F}^{Y}_t, S_t] = \frac{p(S_t,t)}{q(S_t,t)} \quad a.s.
    \end{equation}
    
    As we observed in \cref{subsec:PFs}, the importance weights are not unique.
    It is now also evident from \cref{eq:weight_Radon-Nikodym} that there are many solutions which are not necessarily optimal.
    A large class of suboptimal choices for $W_t$ is of the form $(d\tilde P_t/d\tilde Q_t)(Z_t,S_t)$, where $\tilde P_t$ and $\tilde Q_t$ are respectively the joint conditional (on $\mathcal{F}^Y_t$) distributions of $(Z_t,X_t)$ and $(Z_t,S_t)$ with $Z_t$ being some arbitrary $\mathcal{F}_t$-measurable process such that the Radon--Nikodym derivative exists, and the second marginals of $\tilde P_t$ and $\tilde Q_t$ agree, respectively, with $P_t$ and $Q_t$.
    For example, we can let $Z_t=S_0$, which takes some information from the past into account.
    Also the weights in the BPF \cref{eq:BPF_W} are of this form, as by the disintegration of measures and independence of $(X_t)_{t\geq 0}$ and $(Y_t)_{t\geq 0}$ under the measure $\mathbb{Q}$,\footnote{i.e., $\mathbb{P}\big |_{\mathcal{F}^Y_t\vee\sigma(X_t)}=\mathbb{P}\big |_{\mathcal{F}^Y_t}\otimes\mathbb{P}_{X_t|\mathcal{F}^Y_t}$ and $\mathbb{Q}\big |_{\mathcal{F}^Y_t\vee\sigma(X_t)}=\mathbb{Q}\big |_{\mathcal{F}^Y_t}\otimes\mathbb{Q}\big |_{\sigma(X_t)}$.} we have $\mathbb{E}_{\mathbb{Q}}[\tfrac{d\mathbb{P}}{d\mathbb{Q}}|\mathcal{F}_t]= \mathbb{E}_{\mathbb{Q}}[\tfrac{d\mathbb{P}}{d\mathbb{Q}}|\mathcal{F}^Y_t]\times\tfrac{dP_t}{dQ_t}(X_t)$, and hence the conditional expectation of the (normalized) Radon-Nikodym derivative of the full measures (after replacing $X_t$ by $S^{\text{\tiny B}}_t$) reduces to $\smash{\tfrac{dP_t}{dQ_t}(S^{\text{\tiny B}}_t)}$.
    
    In the FPF, we have by construction that $P_t=Q_t$ and therefore $W^{\text{\tiny F}}_t=1$, which is also the optimal choice---in fact, it is globally optimal as it achieves $\text{Var}[W^{\text{\tiny F}}_t]=0$.
    Accordingly, this result reconciles the BPF with the FPF and all other filters for which $W_t$ as a pathwise Radon--Nikodym derivative does not make sense.
    \cref{thm:weights} says that the conditional expectation of $W_t$ can always be interpreted as a Radon--Nikodym derivative and is a version of (i.e., a.s. equal to) the density of the filtering distribution $P_t$ with respect to the particle distribution $Q_t$.
    It also shows that in the BPF, $W^{\text{\tiny B}}_t$ is different from the optimum $W^*_t$ given that we fix the particle dynamics to the prior dynamics.
    This is not surprising in view of the well-known degeneracy problem of the BPF (see \cite{Surace2019} and the references therein).
    
\subsection{The unifying theorem}\label{subsec:unifying}
    Considering the pair process $(S_t,W_t)$, which represents particles and their importance weights, respectively, as explained in \cref{subsec:PFs}, we are trying to find conditions on the dynamics of $(S_t,W_t)$ that allow the particle system to target the filtering distribution, as defined in \cref{def:targets}. 
    Here we make an ansatz that the stochastic dynamics of the particle system takes the following general form:
    \begin{align}
        dS_t & = u (S_t,t)\,dt + k(S_t,t)\,dY_t + v(S_t,t)\, dB_t, \label{eq:HPF_dS}\\
        dW_t & = W_t \big( \gamma (S_t,t)\,dt + \varepsilon(S_t,t)\,dY_t + \zeta(S_t,t)\,dB_t \big), \label{eq:HPF_dW}
    \end{align}
    where $B_t$ is a BM. From this starting point, our goal is to find the conditions that the unknown functions $\{u,k,v,\gamma,\varepsilon,\zeta\}$ should satisfy. 
    To derive the results, certain assumptions are required as listed below:
    \begin{enumerate}[label=(\roman*),leftmargin=*]
        \item\label{asmp:SDEs} Regularity conditions ensure well-posedness of the system of SDEs \cref{eq:HPF_dS}--\cref{eq:HPF_dW}, together with \cref{eq:hidden}--\cref{eq:observation} (see, e.g., \cite[Theorem 6.30]{Klebaner2005}).
        \item\label{asmp:initial} The initial condition is $S_0 \sim P_0$ and $W_0=1$.
        \item\label{asmp:BM} The BM $B_t$ is independent of \{$(B^{\text{\tiny X}}_t)_{t\geq 0}$, $(B^{\text{\tiny Y}}_t)_{t\geq 0}$, $X_0$, $S_0$\}.
        \item\label{asmp:functions_S} $u \in C^{1,0}$ and $k,v \in C^{2,0}$.
        \item\label{asmp:functions_W} $\gamma, \varepsilon, \zeta \in C^{1,0}$.
    \end{enumerate}
    The second key result of our work is the following:
	\begin{theorem}[unifying theorem]\label{thm:unifying}
    	Consider the filtering problem \cref{eq:hidden}--\cref{eq:observation} with filtering density $p(x,t)$, which satisfies the KSE \eqref{eq:KSE_dis}. 
    	Let $(S_t, W_t)$ be a pair process, representing particles and their importance weights, respectively, that evolves according to the dynamics \cref{eq:HPF_dS}--\cref{eq:HPF_dW} under the assumptions \ref{asmp:SDEs}--\ref{asmp:functions_W}. 
    	Then if the particle system described by $(S_t,W_t)$ targets the filtering distribution for all $0 < t< T$, according to \cref{def:targets}, the functions \{$u(x,t), k(x,t),v(x,t),\gamma(x,t),\varepsilon(x,t),\zeta(x,t)$\} satisfy the following equations for all $0 \leq t<T$:
    	\begin{gather}
    	    \frac{\partial}{\partial x} ( k p ) + ( h - \hat{h}_t - \varepsilon ) p = 0, \label{eq:PDE_dY} \\
        	\frac{1}{2} \frac{\partial^2}{\partial x^2} \big( ( k^2 + v^2 - g^2 ) p \big) - \frac{\partial}{\partial x} \big( ( u +  k \varepsilon + v \zeta - f)  p \big) + (  (h - \hat{h}_t)\hat{h}_t + \gamma ) p = 0; \label{eq:PDE_dt}
    	\end{gather}
    	in addition, the functions \{$\varepsilon(x,t),\gamma(x,t)$\} have zero mean under the filtering distribution for all $0 \leq t<T$, that is,
    	\begin{equation}\label{eq:zero-mean}
    	    \int_{\mathbb{R}} \varepsilon(x,t) p(x,t) dx = 0, \quad\, \int_{\mathbb{R}} \gamma(x,t) p(x,t) dx = 0.
    	\end{equation}
    \end{theorem}
    
    The proof appears in \cref{prf:unifying}.
    In short, the results follow from the fact that the terms multiplying $dY_t$ and $dt$ on both sides of $d\mathbb{E}[ W_t\phi(S_t) | \mathcal{F}^{Y}_t ] = d\mathbb{E}\big[\phi(X_t) | \mathcal{F}^{Y}_t\big]$ should be equal a.s., regardless of $\phi$.
    In particular, considering $\phi=1$ leads to the last statement \cref{eq:zero-mean}, while considering the class of compactly supported test functions $\phi\in C^2_k$ yields the system of ODEs \cref{eq:PDE_dY}--\cref{eq:PDE_dt}, which become PDEs in higher dimensions. 
    Note that through these equations, the coefficients in the particle system dynamics depend on the targeted distribution. 
    This means that the system of SDEs \cref{eq:HPF_dS}--\eqref{eq:HPF_dW} are of McKean--Vlasov type.
    Recall that based on formula \cref{eq:targets}, the distribution targeted by a particle system is obtained by evaluating the left-hand side of this formula with an indicator function as a test function $\phi$. 
    
    Compared to the BPF \cref{eq:BPF_dS}--\cref{eq:BPF_dW} and FPF \cref{eq:FPF_dS}--\cref{eq:FPF_dW}, our ansatz \cref{eq:HPF_dS}--\cref{eq:HPF_dW} is clearly more general in several aspects, for example the presence of $dB_t$ and $dY_t$ in both dynamics. 
    Thus, we may refer to this approach as “hybrid particle filter.”
    As we shall show in \cref{subsec:dW_optimal}, the presence of $dB_t$ in the weight process actually decreases its variance. 
    Here the particle dynamics is still supposed to not involve $W_t$ explicitly and the coefficients in the weight dynamics are intentionally chosen to be linear in $W_t$. 
    This choice has an advantage, as can be seen in the proof of the theorem.
    Specifically, it allows us to use the targeting assumption in order to convert conditional expectations appearing in $d \mathbb{E}[W_t \phi(S_t)|\mathcal{F}^{Y}_t]$ to posterior expectations.
    
    The theorem above, while providing only necessary conditions for targeting the filtering distribution, sheds light on the freedom in choosing the coefficients of the particle and weight dynamics. 
    It is easy to verify that the setting \{$u=f$, $k=0$, $v=g$, $\zeta=0$\} yields the BPF \cref{eq:BPF_dS}--\cref{eq:BPF_dW}. 
    We demonstrate in \cref{subsec:class} that the FPF also satisfies the necessary conditions. 
    Note that the first equation \cref{eq:PDE_dY} is similar to the gain equation \cref{eq:FPF_K} in the FPF except the extra term $\varepsilon p$, which arises here due to the nonzero weight dynamics. 
    This freedom might help us compensate for the gain estimation errors with weight dynamics, as will be outlined in \cref{subsec:gain_estimation}.
    
    We close this subsection by pointing out that our results also hold in the unconditional setting, i.e., for a particle system targeting the solution of the Fokker--Planck equation. 
    In particular, if we set $h=0$, then the observation process $Y_t$ does not provide any information about $X_t$ and the KSE \cref{eq:KSE_dis} reduces to the Fokker--Planck equation. 
    The corollary below follows immediately from \cref{thm:unifying} and shows that even in this setting, there exists intrinsic freedom in constructing the dynamics of the particle system while keeping its distribution invariant.
    \begin{corollary}\label{crl:unifying_unconditional}
        Consider the stochastic process $X_t$ satisfying the SDE \cref{eq:hidden}, and let $\bar{p}(x,t)$ denote the probability density function of $X_t$, which satisfies the Fokker--Planck equation. Let $(S_t, W_t)$ be a pair process, representing particles and their importance weights, respectively, that evolves according to the dynamics below:
        \begin{align}
            dS_t &= u (S_t,t)\,dt + v(S_t,t)\, dB_t, \\
            dW_t &= W_t \big( \gamma (S_t,t)\,dt + \zeta(S_t,t)\,dB_t \big),
        \end{align}
        under the assumptions \ref{asmp:SDEs}--\ref{asmp:functions_W}, if applicable. 
        Then if the particle system described by $(S_t,W_t)$ targets $\bar{p}$, i.e., $\mathbb{E}[W_t \phi(S_t)] = \mathbb{E}[\phi(X_t)]$ holds for any measurable function $\phi$ with $\mathbb{E}[|\phi(X_t)|]<\infty$, for all $0 < t< T$, the functions \{$u(x,t),v(x,t),\gamma(x,t),\zeta(x,t)$\} satisfy the following equation for all $0 \leq t<T$:
    	\begin{equation} \label{eq:PDE_dt_FPE}
             \frac{1}{2} \frac{\partial^2}{\partial x^2} \big( ( v^2 - g^2 ) \bar{p} \big) - \frac{\partial}{\partial x} \big( ( u + v \zeta - f)  \bar{p} \big) +  \gamma \bar{p} = 0;
    	\end{equation}
    	in addition, the function $\gamma(x,t)$ satisfies $\int_{\mathbb{R}} \gamma(x,t) \bar{p}(x,t) dx = 0 $ for all $0 \leq t<T$.
    \end{corollary}
    
\subsection{A class of particle filters}\label{subsec:class}
    Our goal now is to introduce a class of PFs within the result of \autoref{thm:unifying} that encompasses the BPF as well as the FPF.
    This demonstrates how these seemingly different methods can be derived from the same framework. 
    Observe that in \cref{eq:PDE_dY}, $\varepsilon=h-\hat{h}_t$ implies $k=0$ while $\varepsilon=0$ yields \cref{eq:FPF_K}. 
    Thus, a choice for $\varepsilon$ that linearly interpolates between $h-\hat{h}_t$ and $0$ makes it possible to have a smooth transition between the BPF and FPF, though it does not simplify the gain equation. 
    The next proposition states the result.
    
    \begin{proposition}\label{prop:class}
        Under the same assumptions as in \autoref{thm:unifying}, and assuming that \cref{eq:PDE_dY} holds, a particular solution to \cref{eq:PDE_dt} is given by the following class:
    	\begin{align}
    	    \varepsilon\; & = \eta (h-\hat{h}_t ) + \tilde{\varepsilon}, \label{eq:class_varepsilon}  \\
    	    \gamma\; & = - (\alpha + \eta - \alpha \eta)  (h - \hat{h}_t) \hat{h}_t - (1-\alpha)  \tilde{\varepsilon} \hat{h}_t , \label{eq:class_gamma} \\
    	    v^2 & = g^2 - \beta k^2, \label{eq:class_v} \\
    	    u\; & =  f - \tfrac{1}{2} k \big(\vartheta_1 h + \vartheta_2  \hat{h}_t + (1+\beta)\tilde{\varepsilon}-(1-\beta) \tfrac{\partial}{\partial x} k  \big) - v \zeta, \label{eq:class_u}
    	\end{align}
    	where $\tilde{\varepsilon} \in C^{1,0}$ is an arbitrary function with zero mean under the filtering distribution such that $\varepsilon$ meets the requirements of \ref{asmp:SDEs}, \{$\alpha$, $\beta$, $\eta$\} are free parameters in $\mathbb{R}$, either constant or continuously time-varying, and $\vartheta_1,\vartheta_2$ are defined as follows:
    	\begin{equation}
    	    \vartheta_1 :=1-\beta+\eta+\beta\eta , \quad\,\, \vartheta_2 := 1+\beta-\eta-\beta\eta-2\alpha.
    	\end{equation}
    \end{proposition}
    
    The proof is simple, and it is given in \cref{prf:class}. 
    In particular, if in the class above, we fix \{$\tilde{\varepsilon}=0$, $\zeta=0$, $\beta=0$\} and let the others \{$\alpha$, $\eta$\} be free, we obtain a subclass that interpolates the BPF (when $\eta=1$) and the FPF (when $\eta=0$, $\alpha=0$). 
    The parameter $\eta$ can be interpreted as the “observation parameter,” which determines how much the observation process is incorporated into the particle or weight dynamics.
    We refer to $\alpha$ as the “drift parameter,” which only appears in drift functions $u,\gamma$.
    Last, we call $\beta$ the “diffusion parameter”
    since it controls the magnitude of the diffusion coefficient $v$.
    Among these, $\eta$ is the most relevant one as far as filtering is concerned.
    
    It should be noted that not all parameter choices for \{$\alpha$, $\beta$, $\eta$\} give rise to a “practical” particle filter. 
    In practice, one has to check other criteria, for example the nondegeneracy of the particle distribution and the stability of the system. 
    This can be seen explicitly in the linear-Gaussian case below. 
    
     \begin{example}[linear-Gaussian case, continued]
     Consider the three-parameter \{$\alpha$, $\beta$, $\eta$\} subclass of particle filters given by \cref{prop:class} after setting $\tilde{\varepsilon}=0$ and $\zeta=0$ for the linear-Gaussian setting \cref{eq:linear_problem} with the filtering distribution \cref{eq:posterior_mean}--\cref{eq:posterior_var}. Then the (unweighted) particle distribution reads as $S_t|\mathcal{F}^Y_t \sim \mathcal{N}(\mu_t,\rho_t)$, where
        \begin{align}
		d\mu_t  &= a\mu_t\,dt + c\hat{\rho}_t (1-\eta) \big(dY_t -  c [ \tfrac{1}{2}\vartheta_1 \mu_t + \tfrac{1}{2} \vartheta_2 \hat{\mu}_t ] \,dt \big), \label{eq:particle_mean} \\
		d\rho_t &= b^2\,dt + 2a\rho_t\,dt -c^2 \hat{\rho}_t (1-\eta) [ \vartheta_1 \rho_t + (\beta-\beta\eta) \hat{\rho}_t ]\,dt. \label{eq:particle_var}
		\end{align}
     The derivation is briefly explained in \cref{prf:q}.
     Observe how particle distribution interpolates between the prior distribution (when $\eta=1$) and the posterior distribution (when $\eta=0$, $\alpha=0$).
     It is easy to confirm that in the latter case, where $\beta$ remains as the only free parameter, the PF resulting from \cref{prop:class} corresponds to the unweighted linear PF stated in \cite[equation 17]{Abedi2019}, with one BM, if it is indeed rewritten in terms of $v$.
     Notably, $v=0$ recovers the deterministic linear FPF introduced in \cite{Taghvaei2016}.
     
     As the right-hand side of \cref{eq:class_v} must be nonnegative, we deduce that $\beta (1-\eta)^2 \leq \smash{b^2/(c^2 \hat{\rho}_t^2) }$.
     Moreover, in the case where $\smash{\hat{\rho}_\infty:=\lim_{t\to\infty}\hat\rho_t<\infty}$, if we want to avoid that the variance grows exponentially, we need to restrict $\eta$ such that $\smash{\eta^2 < 1 - 2a /(c^2 \hat{\rho}_\infty)}$.
    Constraints like these become important in numerical implementations.
    \end{example}

\subsection{Stochastic differential of the optimal weight}\label{subsec:dW_optimal} 
    We saw in \cref{subsec:weights} that weights for a fixed particle distribution $Q_t$ are not unique and indeed $W^*_t := \frac{dP_t}{dQ_t}(S_t)$ is the one that minimizes the variance.
    In the context of the unifying theorem, which assumes additional constraints regarding the time-evolution of the particle system, this nonuniqueness means that if we fix the particle dynamics, there are many possibilities for the weight dynamics.
    In other words, if we fix the functions \{$u,k,v$\}, we are then left with three unknowns \{$\gamma,\varepsilon,\zeta$\} but only two equations \cref{eq:PDE_dY}--\cref{eq:PDE_dt} to constrain them.
    Here we demonstrate that for each choice of the dynamics of $S_t$ according to \cref{eq:HPF_dS}, the dynamics of the optimal weight $W^*_t$ also takes the form of \cref{eq:HPF_dW} whose coefficients denoted by \{$\gamma^*,\varepsilon^*,\zeta^*$\} are given by the proposition below.

	\begin{proposition}\label{prop:dW_optimal}
	    Fix the particle dynamics \cref{eq:HPF_dS} with $S_0 \sim P_0$ under the assumptions \ref{asmp:BM}--\ref{asmp:functions_S}.
	    Then the optimal weight $W^*_t$ from \cref{thm:weights} satisfies the SDE
    	\begin{equation}\label{eq:dW_t^*}
    	    dW_t^* = W_t^*\big( \gamma^* (S_t,t)\,dt + \varepsilon^*(S_t,t)\,dY_t + \zeta^*(S_t,t)\,dB_t\big),
    	\end{equation}
	    where $\varepsilon^*$, $\gamma^*$ are $P_t$-a.s. unique solutions to equations \cref{eq:PDE_dY}--\cref{eq:PDE_dt} after setting
		\begin{equation}\label{eq:zeta^*}
	        \zeta(x,t) = \zeta^*(x,t) := v(x,t) \tfrac{\partial}{\partial x} \big(\log \tfrac{p(x,t)}{q(x,t)}\big).
	    \end{equation}
	    In addition, $\varepsilon^*$, $\gamma^*$ satisfy the condition \cref{eq:zero-mean}.
	\end{proposition}
	
	The proof, which appears in \cref{prf:dW_optimal}, has a  straightforward idea, but requires lengthy calculations to find the stochastic differential of $\smash{ W_t^* := \tfrac{p(S_t,t)}{q(S_t,t)}}$ based on the Kunita--It\^o--Wentzell formula (see \cite[Theorem 1.1]{Kunita1981}), as the functions $p(x,t)$ and $q(x,t)$ satisfy SPDEs and their ratio is evaluated at $S_t$, which solves an SDE.
	
	The presence of the noise term $\zeta^*(S_t,t)\,dB_t$ in \cref{eq:dW_t^*} might appear counterintuitive in view of the goal to minimize the variance of $W_t$.
	Indeed, adding an independent BM term would only increase the variance.
	However, the BM appearing in \cref{eq:dW_t^*} is the same as that driving the process $S_t$, introducing correlation between $W_t^*$ and $S_t$, which consequently helps $W_t^*$ to achieve the minimal variance possible given the dynamics of $S_t$. In particular, if we restrict the particle evolution to the prior dynamics (i.e., $u=f$, $k=0$, $v=g$), then \cref{prop:dW_optimal} results in
	\begin{equation}\label{eq:dW_t^*_BPF}
	    dW_t^*  = W_t^* (h(S_t,t) - \hat{h}_t)\, (dY_t-\hat{h}_t dt) + W_t^* (\lambda^{*} (S_t,t) \, dt + \zeta^* (S_t,t) \, dB_t),
	\end{equation}
	where $\zeta^*$ is given by \cref{eq:zeta^*} with $v=g$ and $\lambda^{*}$ is the remaining drift coefficient:
	\begin{equation}\label{eq:lambda^*_BPF}
	    \lambda^{*}(x,t) : =  \tfrac{1}{p(x,t)} \tfrac{\partial}{\partial x} \big( g^2(x,t) p(x,t) \tfrac{\partial}{\partial x} (\log \tfrac{p(x,t)}{q(x,t)}) \big).
	    \end{equation}
	Notice that the presence of an additional term in \cref{eq:dW_t^*_BPF} compared to the weight process in the BPF \cref{eq:BPF_dW} indicates that the importance sampling used in the BPF is not optimal.
	Nonetheless, this correction term is not easy to evaluate in practice since the functions $\lambda^{*},\zeta^*$, which involve the densities $p,q$ explicitly, must be estimated from samples. In the linear-Gaussian case, however, the analytical computation is possible.
	
	\begin{example}[linear-Gaussian case, continued]
         Let the particle process evolve as the prior dynamics, and denote by $\mathcal{N}(\mu_t,\rho_t)$ its  (unweighted) distribution. Then \cref{eq:dW_t^*_BPF}--\cref{eq:lambda^*_BPF} can be computed analytically and yield the following PF:
    	\begin{align}
        	dS_t & = a S_t \, dt + b \, dB_t, \\
        	dW_t^*
        	 & = W_t^* (cS_t - c\hat{\mu}_t) (dY_t - c\hat{\mu}_t dt ) \\
        	 & + W_t^* b^2 \big( \tfrac{1}{\rho_t}-\tfrac{1}{\hat{\rho}_t} + (\tfrac{S_t-\hat{\mu}_t}{\hat{\rho}_t})^2 -
        	 \tfrac{(S_t-{\mu}_t) (S_t-\hat{\mu}_t)}{\rho_t\hat\rho_t} \big) \, dt \nonumber \\
        	 & + W_t^* b \big( \tfrac{S_t-{\mu}_t}{{\rho}_t} - \tfrac{S_t-\hat{\mu}_t}{\hat{\rho}_t}  \big) \, dB_t. \nonumber
    	\end{align}
    \end{example}

\section{Applications}\label{sec:applications}
\subsection{Compensating for gain approximation with weight dynamics}\label{subsec:gain_estimation}
    The gain function $K$ in the FPF is a solution to \cref{eq:FPF_K}, which is fixed; i.e., it only depends on the model. 
    In the unifying theorem, however, the function $k$, which has a role similar to that of $K$, is a solution to \cref{eq:PDE_dY}, which now has a term $\varepsilon$ which can be chosen freely. 
    Observe that $\frac{\partial}{\partial x} \big( (k-K) p \big) = \varepsilon p,$ which means that any deviation of $k$ from $K$ in the particle dynamics is compensated by $\varepsilon$ from the weight dynamics such that the particle system can eventually target the filtering distribution. 
    Recall from condition \cref{eq:zero-mean} that the mean of $\varepsilon$ is zero under the filtering distribution. 
    The variance of $\varepsilon$ under the filtering distribution can serve as an “error,” measuring the difference between $k$ and $K$. 
    Moreover, smaller values of $\varepsilon$ are preferable because to keep the weights close to unity, the coefficients in the weight dynamics, including $\varepsilon$, should remain as close to zero as possible. 
    
    There are now two approaches to exploiting this freedom in \cref{eq:PDE_dY}: (1) to first set $\varepsilon$ and then solve the equation for $k$, and (2) to set $k$ and find $\varepsilon$ afterwards. 
    In (1), the presence of $\varepsilon$ allows us to modify the equation to some extent, which might help us to use a simpler gain estimation method and compensate for it by an appropriate weight dynamics.
    It will be for future research to explore this possibility.
    In (2), which we explore in more detail, instead of setting $k$ explicitly, we may restrict ourselves to a specific class of functions denoted by $\mathcal{K}$ and pose the following variational problem:
    \begin{align}\label{eq:gain_variational}
    	\begin{aligned}
        	& \underset{k}{\text{minimize}}
        	& & \mathcal{I}[k] := \int_{\mathbb{R}} \varepsilon^2(x,t) p(x,t) dx \\
        	& \text{subject to}
        	& & k \in \mathcal{K}, \quad \varepsilon p = ( h - \hat{h}_t) p + \tfrac{\partial}{\partial x} ( k p ), 
    	\end{aligned}
    \end{align}
    for a fixed $t$ and $p$. The second constraint above comes from \cref{eq:PDE_dY}, and $k$ must also satisfy the regularity assumptions required for \cref{thm:unifying}.
    The objective functional $\mathcal{I}$ can be interpreted as the square of the Fisher--Rao norm $\|\dot p\|_{\text{FR}}^2$ of the fictional change in the distribution corresponding to the continuity equation $\dot p + \frac{\partial}{\partial x} \big( (k-K) p \big)=0$ (here we omit the time dependence of $k,K$ and the dot relates to a fictional time associated with the flow of the fixed vector field $k-K$).
    This contrasts with the use of the infinitesimal Wasserstein-2 or Otto's norm $\|\dot p\|_{W_2}^2=\int (k(x)-K(x))^2p(x)dx$ (see \cite{Lott2008}).
    As an example, we consider the case when $\mathcal{K}$ is the set of constant functions.
    In this case, the minimization of $\|\dot p\|_{W_2}^2$ over $\mathcal{K}$ gives rise to the standard constant gain approximation in the FPF (see \cite[Example 2 and Remark 5]{FPF2016}), and analogously, the problem \eqref{eq:gain_variational} admits a simple analytical solution, which we present in the proposition below and call Fisher-optimal constant gain for the aforementioned reasons.
    
    \begin{proposition}[Fisher-optimal constant gain approximation] \label{prop:optimal_constant_gain}
        Consider \cref{eq:PDE_dY} in \cref{thm:unifying}, and let $\mathcal{K}$ be the set of functions independent of $x$ with elements $\bar{k}(t)$, i.e., $\bar{k}(t) \in \mathcal{K}=\{ k(x,t): k'(x,t) =0 \,\, \forall x \in \mathbb{R} \}$. 
        Then the solution $\bar{k}^*(t)$ to the variational problem \cref{eq:gain_variational} at time $t$ is
    	\begin{equation}\label{eq:optimal_k}
    	    \bar{k}^*(t) =  \frac{\mathbb{E}[ h'(X_t,t) | \mathcal{F}^{Y}_t] }{ \mathbb{E}[ \psi^2(X_t,t) | \mathcal{F}^{Y}_t] },
    	 \end{equation}
    	 and the corresponding $\varepsilon$, which can be determined only up to $P_t$-null set, is $\varepsilon^*(x,t) := h(x,t) - \hat{h}_t +  \bar{k}^*(t) \, \psi(x,t)$, where we have defined $\psi(x,t) := {p'(x,t)}/{p(x,t)}$ over the interior of the support of $p$ and $\psi(x,t) :=0$ otherwise.
    \end{proposition}
    
    The proof is given in \cref{prf:gain}. Note that $\mathbb{E}[ \psi^2(X_t,t) | \mathcal{F}^{Y}_t]$ is itself a Fisher information, namely that of the 1-parameter model $p_{\theta}(x,t)=p(x-\theta,t)$ at $\theta=0$.
    
    \begin{example}[linear-Gaussian case, continued]
    It is easy to verify that in the linear-Gaussian case \cref{eq:linear_problem}, $\bar{k}^*(t)$ from \cref{eq:optimal_k} yields the Kalman gain, which in turn makes $\varepsilon^*(x,t)$ and thus the objective functional zero. Observe that
	\begin{align}
	     h'(x,t) = c \quad & \implies \quad \mathbb{E}[ h'(X_t,t) | \mathcal{F}^{Y}_t]= c, \\
	     \psi(x,t) = -(x-\hat{\mu}_t)/\hat{\rho}_t \quad & \implies \quad \mathbb{E}[ \psi^2(X_t,t) | \mathcal{F}^{Y}_t]  = 1/ \hat{\rho}_t,
	\end{align}
    which gives the Kalman gain $\bar{k}^*(t)= c \hat{\rho}_t $. 
    \end{example}
    
    In general (i.e., the nonlinear or non-Gaussian case), $\psi(x,t)$ and the Fisher information $\mathbb{E}[ \psi^2(X_t,t) | \mathcal{F}^{Y}_t]$ need to be estimated from samples.
    Estimators for these quantities can be found, e.g., in \cite{Bhattacharya1967, Cao2020, Fabian1973,  Schuster1969}.
    The proposition above can be generalized to richer classes of functions $\mathcal{K}$, e.g., functions of the form $k(x,t) = \sum_{j} a_j(t) \varphi_j (x,t)$, where $\{\varphi_j\}_j$ is an appropriate set of basis functions and $\{a_j\}_j$ are some coefficients. 

\subsection{Providing freedom to alleviate weight degeneracy}\label{subsec:weight_degeneracy}
    Weight decay is a major issue among the weighted PFs.
    As discussed in \cref{subsec:PFs}, in order to have less weight degeneracy, $\mathbb{E}[W_t^2]$ should remain small.
    In \cref{subsec:weights,subsec:dW_optimal}, we derived the optimal importance weight $W_t^*$ and its stochastic dynamics $dW_t^*$ under the constraint that particle distribution is given, i.e., particle dynamics is fixed. 
    Here we would like to derive a general formula for $d\mathbb{E}[W_t^2]$ corresponding to the weight dynamics of form \cref{eq:HPF_dW} in terms of \{$\gamma, \varepsilon,\zeta$\}.
    By Itô's formula, we have 
    \begin{equation}\label{eq:dW^2}
	    dW_t^2 = 2 W_t^2 \big( \gamma (S_t,t) dt +  \varepsilon(S_t,t) dY_t +  \zeta(S_t,t) dB_t + \tfrac{1}{2}  \varepsilon^2(S_t,t) dt + \tfrac{1}{2}   \zeta^2(S_t,t) dt \big).
	\end{equation}
	Applying Fubini’s theorem to the integral form of SDE above would suffice for deriving $d\mathbb{E}[W_t^2]$.
	However, the presence of $W_t^2$, which multiplies functions of $S_t$, as well as $dY_t$, which in turn depends on the hidden process $X_t$, make the analysis of $d\mathbb{E}[W_t^2]$ complicated.
	It is interesting to note that studying $d\log(W_t)$ does not have the first issue because it no longer involves $W_t$ explicitly.
	Observe that by Itô's formula
	\begin{equation}\label{eq:dlogW}
	    d\log(W_t) = \gamma (S_t,t) dt + \varepsilon(S_t,t) dY_t   +  \zeta(S_t,t) dB_t - \tfrac{1}{2} \varepsilon^2(S_t,t) dt - \tfrac{1}{2} \zeta^2(S_t,t) dt.
	\end{equation}
	To overcome the second issue, it turns out that if we use the fact that the innovation process is a BM, we will get a more useful formula as follows.

    \begin{proposition}\label{prop:weight_degeneracy}
        Consider the filtering problem \cref{eq:hidden}--\cref{eq:observation}. Let $(S_t, W_t)$ be a pair process such that $W_t$ solves the SDE \cref{eq:HPF_dW} under the assumptions \ref{asmp:initial}, \ref{asmp:BM}, and \ref{asmp:functions_W} while $S_t$ is a progressively measurable process with respect to $\mathcal{F}^B_t \vee \mathcal{F}^Y_t \vee \sigma(S_0)$. 
        Suppose that $h(X_{\cdot},\cdot) \in \mathbb{L}^1(0,t)$  and $\mathbb{E}[W_t^2],\mathbb{E}[|\log W_t|] < \infty$ for all $t\geq 0$.
        Then
        \begin{gather}
        	d \mathbb{E}[W_t^2]
        	 = 2 \mathbb{E}\big[W_t^2 \big( \gamma (S_t,t) +  \varepsilon(S_t,t) \hat{h}_t  + \tfrac{1}{2} \varepsilon^2(S_t,t) +  \tfrac{1}{2} \zeta^2(S_t,t) \big) \big]\,dt, \label{eq:dE[W^2]_general} \\
        	d \mathbb{E}[\log (W_t)]  = \mathbb{E} \big[ \gamma (S_t,t) + \varepsilon(S_t,t) \hat{h}_t - \tfrac{1}{2} \varepsilon^2(S_t,t)  - \tfrac{1}{2} \zeta^2(S_t,t) \big] \, dt. \label{eq:dE[log W]_general}
        \end{gather}
    \end{proposition}
    
    The proof appears in \cref{prf:dEW2}.
    Notice that due to the nonnegativity of $\text{Var}[W_t]$, we have $\mathbb{E}[W_t^2] \geq 1$ and due to $\log(x) \leq x -1 $, we have $ \mathbb{E}[\log(W_t)] \leq 0$.
    It can be shown (by counterexample) that $\mathbb{E}[\log(W_t)]$ is not necessarily a proper measure for degeneracy.
    However, the formulas above are insightful for investigating the possibility of exploiting the freedom given by the terms $\gamma,\varepsilon,\zeta$ in order to find other steady-state solutions for $\mathbb{E}[W_t^2]$ besides the minimum-variance solution in \cref{prop:dW_optimal}.
    We defer this investigation to future research, but to illustrate the usefulness of the formulas above, we apply them to the weights in the BPF \cref{eq:BPF_dW}:
    \begin{gather}
    	d \mathbb{E}[{W_t^{\text{\tiny B}}}^2]
    	 = \mathbb{E}\big[{W_t^{\text{\tiny B}}}^2( h(S_t^{\text{\tiny B}},t) - \hat{h}_t )^2 \big]\,dt, \label{eq:dE[W^2]_BPF} \\
    	d \mathbb{E}[\log (W_t^{\text{\tiny B}})]  = - \tfrac{1}{2} \mathbb{E} \big[ ( h(S_t^{\text{\tiny B}},t) - \hat{h}_t )^2   \big] \, dt. \label{eq:dE[ln W]_BPF}
    \end{gather}
    Equation \cref{eq:dE[W^2]_BPF} implies that in the BPF, we always have $d \mathbb{E}[{W_t^{\text{\tiny B}}}^2] \geq 0$, and hence the number of effective particles will inevitably decay over time.
    Note also that if the observation process is $m$-dimensional, the drift coefficient on the right-hand side of \cref{eq:dE[W^2]_BPF} is given by $\mathbb{E}[{W_t^{\text{\tiny B}}}^2 \| h(S_t^{\text{\tiny B}},t) - \hat{h}_t \|^2]$, and therefore the time constant of the weight decay will scale with $1/m$ (compare the analysis with \cite[section 3.1.1]{Surace2019}).

\section{Discussion}\label{sec:conclusion}
Existing particle filters fall into two distinct types. 
Unweighted PFs (such as the FPF) assimilate new data by moving around particles while keeping the weights associated to each particle fixed, whereas weighted PFs (such as the BPF) assimilate new data by reweighing particles. 
In this paper, we proposed a unifying framework for these types of PFs. 
Our proposed hybrid filter allows particles to be moved as well as reweighed in response to new observations. 
This gives a lot of freedom on how to design a PF. 
This freedom obviously needs to be constrained in order to make sure that the empirical distribution of weighted particles effectively converges to the filtering distribution (i.e., to be asymptotically exact).
The necessary conditions are summarized by \cref{eq:PDE_dY}--\cref{eq:PDE_dt} in the unifying theorem. 

Even after having constrained the freedom of the particle and weight dynamics to satisfy the targeting condition, there is still substantial freedom that could be exploited. 
An interesting extension of the present work will be to determine whether there exists a “sweet spot” where the strengths of unweighted PFs (i.e., the absence of the weight collapse problem) could be combined with the strengths of simple weighted PFs such as the BPF (where the solution of \cref{eq:PDE_dY}--\cref{eq:PDE_dt} is trivial). 
Practically, this would amount to defining a cost function which combines the cost associated to the severity of the weight decay as well as the cost for computing the solutions of \cref{eq:PDE_dY}--\cref{eq:PDE_dt}. We leave this question for further work. 

Another extension of the present work would be to relax the strong assumptions made in the SDEs \cref{eq:HPF_dS}--\cref{eq:HPF_dW}. 
Indeed, in \cref{eq:HPF_dS}, it is assumed that the particle dynamics does not involve $W_t$ explicitly and \cref{eq:HPF_dW} assumes that the coefficients in the weight dynamics are linear in $W_t$.
In this construction, $S_t$ is the main process, but $W_t$ can be thought of as an auxiliary process driven by $S_t$ (recall \cref{eq:dlogW}) and transforming the distribution of $S_t$ into the filtering distribution.
These assumptions were made to simplify the expression for the necessary conditions in the unifying theorem. 
However, if we relax them, we will obtain even more freedom on the choice of particle and weight dynamics that could be leveraged to increase the chance of obtaining a sweet spot as described in the previous paragraph.
For example, we could enforce a drift term in the weight dynamics that pulls the weight back to unity and compensate for this specific choice by appropriate functions in the particle dynamics. 
This could be interpreted as a smooth resampling procedure, unlike classical resampling, where particles and weights have the undesirable feature of changing abruptly at the resampling times.

Finally, this paper is entirely focused on necessary conditions that the hybrid particle filter needs to satisfy in order to target the filtering distribution.
An obvious next question will be to determine sufficient conditions that need to be satisfied by the model such that \eqref{eq:PDE_dY}--\eqref{eq:PDE_dt} guarantee that the targeting condition is met for an extended period of time (i.e., the converse of \cref{thm:unifying}). 
The main difficulty will be to establish conditions under which the solutions of \eqref{eq:PDE_dY}--\eqref{eq:PDE_dt} make the SDEs \eqref{eq:HPF_dS}--\eqref{eq:HPF_dW} well-posed.
This question has been partly addressed in \cite{Pathiraja2020} for some special cases but remains open in the general case considered here.
Furthermore, additional work is required to establish sufficient conditions under which a PF with a finite number $N$ of particles has a uniformly bounded error. 
Indeed, at the end of \cref{subsec:class}, we saw that in a simple linear-Gaussian case, the variance of the unweighted particles can grow exponentially if a specific parameter is above a given threshold (which is not recognized by the necessary conditions). 
This feature is clearly undesirable when the number of particles is finite since it implies that the number of samples that effectively support the filtering distribution decreases with time, which is similar to what occurs in the BPF.
It would therefore be desirable to derive sufficient conditions that guarantee the stability of the filter. 

\section{Proofs}\label{sec:proofs}
This section presents all the proofs of our results.
In a nutshell, the key results follow directly from the targeting condition \cref{eq:targets} by evaluating different types of test functions $\phi$, specifically, indicator function $\phi (x)=\mathds{1}_{B} (x)$, constant function $\phi (x) =1$, and $\phi \in C^2_k$, which lead to \cref{eq:weight_Radon-Nikodym}, \cref{eq:zero-mean}, and \cref{eq:PDE_dY}--\cref{eq:PDE_dt}, respectively. 

\subsection{Proof of the Radon--Nikodym characterization of the weight}\label{prf:weights}
\begin{proof}[Proof of \cref{thm:weights}] $\,$
    \begin{enumerate}[label=\Roman*.,leftmargin=*]
        \item By the targeting condition \cref{eq:targets}, we have for all integrable functions $\phi$
        \begin{align}
        	\mathbb{E}[W_t \phi(S_t)|\mathcal{F}_t^Y]
        	& = \mathbb{E}\Big[\mathbb{E}[W_t\phi(S_t)|\mathcal{F}_t^Y,S_t] \Big| \mathcal{F}_t^Y \Big] \label{eq_setp_cond} \\
        	& = \mathbb{E}\Big[\phi(S_t) \mathbb{E}[W_t|\mathcal{F}_t^Y,S_t] \Big| \mathcal{F}_t^Y \Big] \\
        	& = \int_{\mathbb{R}} \phi(x) \mathbb{E}[W_t|\mathcal{F}_t^Y,S_t=x] Q_t (dx) \\
        	&= \int_{\mathbb{R}} \phi(x) P_t(dx) = \mathbb{E}[\phi(X_t)|\mathcal{F}_t^Y].
    	\end{align}
    	Let $B \in \mathcal{B}(\mathbb{R})$ be a Borel subset. Take $\phi (x) = \mathds{1}_{B} (x) $ the indicator function of $B$; then
    	\begin{equation}
    	    \int_{B} \mathbb{E}[W_t|\mathcal{F}_t^Y,S_t=x] Q_t (dx) =  P_t(B),
    	\end{equation}
	which shows the first claim \cref{eq:weight_Radon-Nikodym}.
	\item This follows directly from Lemma \ref{lemma:min_var} just below.
    \end{enumerate}
    \end{proof}
    \begin{lemma}\label{lemma:min_var}
        Let $(\Omega, \mathcal{F}, P)$ be a probability space and $W^*$ a $\mathcal{G}$-measurable random variable defined on this space. Suppose $\mathcal{G} \subseteq \mathcal{F}$
    	and $\mathbb{E}[{W^*}^2] < \infty $. Then $W^*$ is the (a.s. unique) solution to the optimization problem 
    	\begin{align}
    	\begin{aligned}
        	& \underset{W}{\text{minimize}}
        	& & \mathbb{E}[W^2] \\
        	& \text{subject to}
        	& & \mathbb{E}[W|\mathcal{G}] = W^* \quad a.s.
    	\end{aligned}
    	\end{align}
	\end{lemma}
    \begin{proof}
        $\mathbb{E} [ W^2]$ can be written as
		\begin{align}
		\mathbb{E} [ W^2 ]
		& = \mathbb{E} [ (W-W^*)^2 ] - \mathbb{E} [ {W^*}^2 ] + 2 \mathbb{E} [ W W^* ] \\
		& = \mathbb{E} [ (W-W^*)^2 ] - \mathbb{E} [ {W^*}^2 ] + 2 \mathbb{E} \big[ \mathbb{E} [W W^*|\mathcal{G}] \big] \\
		& = \mathbb{E} [ (W-W^*)^2 ] - \mathbb{E} [ {W^*}^2 ] + 2 \mathbb{E} \big[ W^* \mathbb{E} [W|\mathcal{G}] \big] \\
		& =  \mathbb{E} [ (W-W^*)^2 ] + \mathbb{E} [ {W^*}^2 ].
		\end{align}
		As $\mathbb{E}[{W^*}^2]$ is fixed, to minimize $\mathbb{E} [ W^2 ]$, we have to minimize $\mathbb{E} [ (W-W^*)^2 ]$. This is attained by any random variable $W_{\text{opt}}$ such that $\mathbb{E} [ (W_{\text{opt}}-W^*)^2 ] = 0.$	Then Chebyshev’s inequality implies that
		\begin{equation}
			\mathrm{Pr} ( |W_{\text{opt}}-W^*| \geq \lambda ) = 0 \quad \text{for all } \lambda > 0.
		\end{equation}
		In other words, 
		$W_{\text{opt}} = W^*$ a.s.
    \end{proof}

\subsection{Proof of the unifying theorem}\label{prf:unifying}
This subsection is devoted to the proof of \cref{thm:unifying}. A key ingredient for the proof is the next lemma, whose first two assertions come from Lemma 2 in \cite{FPF2011}. It allows us to not only interchange conditional expectations with integrals but also adapt the $\sigma$-algebra accordingly, which makes it distinct and stronger from the normal conditional Fubini theorem.
\begin{lemma}\label{lemma:interchange_integrals}
    Take the system of SDEs \cref{eq:HPF_dS}--\cref{eq:HPF_dW} for the pair process $(S_t, W_t)$ under the assumptions \ref{asmp:initial}-\ref{asmp:functions_W}.
    Let $F(x,w,t)$ be an $\mathbb{R}$-valued measurable function such that
    $ F(S_\cdot,W_\cdot,\cdot) \in \mathbb{L}^2(0,t)$. Then
    \begin{align}
        \mathbb{E}\Big[ \int_{0}^{t} F(S_s,W_s,s) \, ds \Big|\mathcal{F}^{Y}_t \Big] & = \int_{0}^{t} \mathbb{E}\big[F(S_s,W_s,s) \big|\mathcal{F}^{Y}_s \big] \, ds, \label{eq:interchange_integral_dt}\\
        \mathbb{E}\Big[ \int_{0}^{t} F(S_s,W_s,s) \, dY_s \Big|\mathcal{F}^{Y}_t \Big] & = \int_{0}^{t} \mathbb{E}\big[ F(S_s,W_s,s) \big|\mathcal{F}^{Y}_s \big] \, dY_s, \label{eq:interchange_integral_dY} \\
        \mathbb{E}\Big[ \int_{0}^{t} F(S_s,W_s,s) \, dB_s \Big|\mathcal{F}^{Y}_t \Big] &= 0. \label{eq:interchange_integral_dB}
    \end{align}
\end{lemma}
\begin{proof}
    The process $F(S_t,W_t,t)$ is $\mathcal{F}_t^Y \vee \mathcal{F}_t^B \vee \sigma(S_0)$-measurable. Statements \cref{eq:interchange_integral_dt} and \cref{eq:interchange_integral_dY} then follow directly from Lemma 2 in \cite{FPF2011}. The last claim \cref{eq:interchange_integral_dB} is similar to the zero-mean property of the Itô integral and has a similar proof, yet it involves conditional expectation. For the Itô integral, we have
    \begin{equation}
        \int_{0}^{t} F(S_s,W_s,s) dB_s =  \lim_{n \to \infty} \sum_{i=0}^{n-1} F(S_{t_i},W_{t_i},t_i) (B_{t_{i+1}} - B_{t_i})
    \end{equation}
    in probability, where $\{ t_i \}_{i=0}^{n}$ is a partition of $[0,t]$ with $\max_i (t_{i+1} - t_{i}) \to  0$ as $n \to \infty$. Consequently, by the dominated convergence theorem for conditional expectations 
    \begin{align}
	    \mathbb{E}\Big[ \int_{0}^{t} F(S_s,W_s,s) dB_s \Big|  \mathcal{F}^{Y}_t \Big]
	    =   
    	\lim_{n \to \infty} \mathbb{E}\Big[ \sum_{i=0}^{n-1} F(S_{t_i},W_{t_i},t_i) (B_{t_{i+1}} - B_{t_i})  \Big|\mathcal{F}^{Y}_t \Big] & \\
    	= 
    	\lim_{n \to \infty} \sum_{i=0}^{n-1} \mathbb{E}\big[ F(S_{t_i},W_{t_i},t_i)\big| \mathcal{F}^{Y}_t \big] \mathbb{E}\big[ B_{t_{i+1}} - B_{t_i} \big] = 0, &
	\end{align}
	where we used the fact that the Brownian motion increment $B_{t_{i+1}} - B_{t_i}$  is independent of $\mathcal{F}^{Y}_t$ as well as the random values $S_{t_i}$ and $W_{t_i}$ at time $t_i$ and its mean is zero.
\end{proof}

Given the lemma above, we are now ready to prove \cref{thm:unifying}.
\begin{proof}[Proof of \cref{thm:unifying}]
    By the targeting assumption for all $0< t<T$ as well as the initial condition, we know that
    \begin{equation}\label{eq:targerts-time}
        \mathbb{E}[W_t \phi(S_t) | \mathcal{F}^{Y}_t] = \mathbb{E}[\phi(X_t) | \mathcal{F}^{Y}_t] \quad\quad  (0 \leq t<T)
    \end{equation}
    holds a.s. for any measurable test function $\phi$ with $\mathbb{E}[|\phi(X_t)|]<\infty$, in particular for any integrable $\phi \in C^2$, which allows us to write an equality for the stochastic differential of these processes in the Itô sense:
    \begin{equation}\label{eq:targerts-time-d}
        d \mathbb{E}[W_t \phi(S_t) | \mathcal{F}^{Y}_t] = d \mathbb{E}[\phi(X_t) | \mathcal{F}^{Y}_t] \quad\quad  (0 \leq t<T).
    \end{equation}
    The right-hand side of expression above is given by \cref{eq:KSE_mom}, which consists of two terms multiplying $dY_t$ and $dt$. The plan is to compute the left-hand side in terms of the unknown functions $\{u,k,v,\gamma,\varepsilon,\zeta\}$ from SDEs \cref{eq:HPF_dS}--\cref{eq:HPF_dW}. It turns out that the left-hand side also consists of two terms multiplying $dY_t$ and $dt$ since we take conditional expectation with respect to $\mathcal{F}^{Y}_t$, and thus terms multiplying $dB_t$ vanish. Two fundamental ODEs (or PDEs in higher dimensions) finally follow from the fact that the terms multiplying $dY_t$ and $dt$ (more precisely, $\,dB_t^{\text{\tiny Y}}$ and $dt$) on each side of \cref{eq:targerts-time-d} are equal a.s. regardless of $\phi$.
    The proof is structured in three steps:
    \begin{enumerate}[label=(\alph*),leftmargin=*]
    	\item We first find the stochastic differential $d (W_t\phi(S_t) )$ from SDEs \cref{eq:HPF_dS}--\cref{eq:HPF_dW} using Itô's formula.
    	\item In order to obtain $d \mathbb{E}[ W_t\phi(S_t) | \mathcal{F}^{Y}_t ]$, we write the result of the previous step in integral form, take the conditional expectation of both sides with respect to $\mathcal{F}^{Y}_t$, use \cref{lemma:interchange_integrals} to interchange conditional expectations with integrals, and finally turn the result back into the differential form. 
    	\item We use the targeting assumption \cref{eq:targerts-time} to convert the conditional expectations that involve $(W_t,S_t)$ into posterior expectations (this becomes possible because of the special form that we take for the stochastic dynamics of the particle system). Finally, we investigate the implications of the equalities resulting from \cref{eq:targerts-time-d} for some class of test functions $\phi$. 
    \end{enumerate}
    Some aspects of our proof are inspired by the usual proof of the Fokker--Planck equation (see, e.g., \cite[Proof of Theorem 5.4]{Sarkka2019}) and the derivation of the FPF \cite{FPF2011, FPF2013}.
    
    \textit{Step} (a). Take any test function $\phi(x) \in C^2$. To obtain the stochastic differential of $W_t \phi(S_t)$ from the SDEs \cref{eq:HPF_dS}--\cref{eq:HPF_dW}, we apply Itô's formula 
	\begin{equation}\label{eq:d_W_phi_S}
		d  \big( W_t \phi(S_t) \big) = W_t  \mathcal{A}_t\phi(S_t) \,dt + W_t \mathcal{B}_t\phi(S_t) \,dY_t + W_t \mathcal{C}_t\phi(S_t) \,dB_t,
	\end{equation}
	where the linear operators \{$\mathcal{A}_t$, $\mathcal{B}_t$, $\mathcal{C}_t$\} have been defined as
	\begin{align}
		\mathcal{A}_t\phi(x) &:=  \gamma(x,t) \phi(x) + \big[u(x,t) + k(x,t)\varepsilon(x,t) + v(x,t) \zeta(x,t)\big] \phi'(x) \\
		           &\quad + \tfrac{1}{2} \big[ k(x,t)^2 + v(x,t)^2\big] \phi''(x) , \nonumber \\
		\mathcal{B}_t\phi(x) &:=  \varepsilon(x,t) \phi(x)  + k(x,t) \phi'(x),    \\
		\mathcal{C}_t\phi(x) &:=  \zeta(x,t) \phi(x) + v(x,t) \phi'(x).
	\end{align}
	
	\textit{Step} (b). Writing the SDE \cref{eq:d_W_phi_S} in integral form and taking the conditional expectation of both sides gives the following:
	\begin{multline}
		\mathbb{E}\big[W_t \phi(S_t)\big|\mathcal{F}^{Y}_t\big] =  \mathbb{E}\big[W_0  \phi(S_0)\big|\mathcal{F}^{Y}_t\big] 
		+ \mathbb{E}\Big[ \int_{0}^{t} W_s \mathcal{A}_s\phi(S_s) ds \Big|\mathcal{F}^{Y}_t \Big] \\
		\quad+ \mathbb{E}\Big[ \int_{0}^{t} W_s \mathcal{B}_s\phi(S_s) dY_s \Big|\mathcal{F}^{Y}_t \Big]
		+ \mathbb{E}\Big[ \int_{0}^{t} W_s \mathcal{C}_s\phi(S_s) dB_s \Big|\mathcal{F}^{Y}_t \Big].
	\end{multline}
	Now apply \cref{lemma:interchange_integrals} to interchange conditional expectations with integrals in the expression above, provided that $ W_\cdot \mathcal{A}_\cdot\phi(S_\cdot), W_\cdot \mathcal{B}_\cdot\phi(S_\cdot), W_\cdot \mathcal{C}_\cdot\phi(S_\cdot) \in \mathbb{L}^2(0,t)$. Moreover, the last term is zero by the third claim of the aforementioned lemma. Then
	\begin{multline}
		\mathbb{E}\big[W_t \phi(S_t)\big|\mathcal{F}^{Y}_t\big] =  \mathbb{E}\big[W_0  \phi(S_0)\big|\mathcal{F}^{Y}_0\big]\\ 
		+ \int_{0}^{t} \mathbb{E}\big[ W_s \mathcal{A}_s\phi(S_s) \big|\mathcal{F}^{Y}_s \big] ds 
		+ \int_{0}^{t} \mathbb{E}\big[ W_s \mathcal{B}_s\phi(S_s) \big|\mathcal{F}^{Y}_s \big] dY_s,
	\end{multline}
	which can be put in differential form as
	\begin{equation}\label{eq:d_E[W_phi_S|F]_1}
		d \mathbb{E}\big[W_t \phi(S_t)\big|\mathcal{F}^{Y}_t\big] = 
		\mathbb{E}\big[ W_t \mathcal{A}_t\phi(S_t) \big|\mathcal{F}^{Y}_t \big] dt + \mathbb{E}\big[ W_t \mathcal{B}_t\phi(S_t) \big|\mathcal{F}^{Y}_t \big] dY_t. 
	\end{equation}
	
	\textit{Step} (c). It becomes clear now why our ansatz \cref{eq:HPF_dS}--\cref{eq:HPF_dW} for the dynamics of the particle system is advantageous. 
	As the particle dynamics does not involve $W_t$ explicitly and the coefficients in the weight dynamics are linear in $W_t$, the arguments of the conditional expectations in \cref{eq:d_E[W_phi_S|F]_1} become linear in $W_t$.
	This allows us to use the targeting assumption \cref{eq:targerts-time}, which holds for any (integrable) measurable function, in order to convert conditional expectations involving $(S_t,W_t)$ to posterior expectations:
	\begin{equation}\label{eq:d_E[W_phi_S|F]_2}
		d \mathbb{E}[W_t \phi(S_t)|\mathcal{F}^{Y}_t] = \mathbb{E}[ \mathcal{A}_t\phi(X_t) |\mathcal{F}^{Y}_t ] dt 
		+ \mathbb{E}[ \mathcal{B}_t\phi(X_t) |\mathcal{F}^{Y}_t ] dY_t. 
	\end{equation}
	Now we have an expression for the left-hand side of \cref{eq:targerts-time-d} in terms of posterior expectations and we recall that the right-hand side is given by \cref{eq:KSE_mom}. As terms multiplying $dt$ and $dY_t$ (more precisely, $\,dB_t^{\text{\tiny Y}}$ and $dt$) on both sides match, we conclude that
	\begin{gather}
	    \mathbb{E}[ \mathcal{A}_t\phi(X_t) |\mathcal{F}^{Y}_t]  =	\mathbb{E} [\mathcal{L}_t\phi(X_t) | \mathcal{F}^{Y}_t] - \mathbb{E} [ \phi(X_t) (h(X_t,t) - \hat{h}_t)\hat{h}_t | \mathcal{F}^{Y}_t ] , \label{eq:match_dt}\\
	    \mathbb{E}[ \mathcal{B}_t\phi(X_t)|\mathcal{F}^{Y}_t]  = \mathbb{E} [ \phi(X_t) (h(X_t,t) - \hat{h}_t) | \mathcal{F}^{Y}_t ] \label{eq:match_dY}
	\end{gather}	
	a.s. The equalities above hold for any $\phi\in C^2$. In particular, $\phi=1$ implies
	\begin{equation}\label{eq:zero-mean_E}
	    \mathbb{E}[ \gamma(X_t,t) |\mathcal{F}^{Y}_t ] = 0,  \quad\, 
	    \mathbb{E}[ \varepsilon (X_t,t) |\mathcal{F}^{Y}_t ] = 0
	\end{equation}
    a.s., which proves the last claim in the theorem.
    
    Additionally, we can use integration by parts in \cref{eq:match_dt}--\cref{eq:match_dY} to rewrite them only in terms of $\phi$ (but not its derivatives). If we further restrict ourselves to compactly supported test functions $\phi \in C^2_k$, boundary terms vanish and we formally have
    \begin{gather}
	   \int_{\mathbb{R}}  \phi(x) \Big[\mathcal{A}_t^{\dag}p(x,t) - \mathcal{L}_t^{\dag}p(x,t) +  (h(x,t) - \hat{h}_t) \hat{h}_t p(x,t) \Big] \, dx = 0 , \label{eq:formal_PDE_dt}\\
	   \int_{\mathbb{R}}  \phi(x) \Big[\mathcal{B}_t^{\dag}p(x,t) - (h(x,t) - \hat{h}_t)p(x,t)\Big] \, dx = 0 , \label{eq:formal_PDE_dY}
	\end{gather}
	where $\mathcal{A}^{\dag}_t$, $\mathcal{B}^{\dag}_t$, and $\mathcal{L}_t^{\dag}$ are the formal adjoints of $\mathcal{A}_t$, $\mathcal{B}_t$, and $\mathcal{L}_t$, respectively, with respect to the Lebesgue measure.
	Since these equations hold for all $\phi \in C^2_k$ and the terms inside the square brackets are continuous in $x$ (given the assumptions \ref{asmp:functions_S} and \ref{asmp:functions_W}), the fundamental lemma in the calculus of variations (see, e.g., \cite[Lemma 1]{Gelfand1963}) implies that the terms inside the square brackets are identically zero on $\mathbb{R}$, which gives two main equations of the theorem, as we will see soon.
    The remainder of the proof will be concerned with calculating the adjoint operators $\mathcal{A}^{\dag}_t$ and $\mathcal{B}^{\dag}_t$.
    
    The integrals we deal with have the general form of $\mathbb{E} [ l(X_t,t) \phi'(X_t) | \mathcal{F}^{Y}_t ]$ or $\mathbb{E} [ l(X_t,t) \phi''(X_t) | \mathcal{F}^{Y}_t ]$, where $l(x,t)$ is a $C^{1,0}$ or $C^{2,0}$ function, respectively. We have
    \begin{equation}
        \mathbb{E} [l(X_t,t) \phi'(X_t) | \mathcal{F}^{Y}_t ] = \int_{\mathbb{R}} \phi' l p \, dx = \phi l p \Big|_{-\infty}^{\infty} -  \int_{\mathbb{R}} \phi ( l p )'  \, dx = -  \int_{\mathbb{R}} \phi ( l p )'  \, dx,
    \end{equation}
    in which the boundary term is zero because $\phi$ has compact support. Similarly, we can use the integration by parts twice and write
    \begin{equation}
        \mathbb{E} [l(X_t,t) \phi''(X_t) | \mathcal{F}^{Y}_t ] = \int_{\mathbb{R}} \phi'' l p\,dx = - \int_{\mathbb{R}}  \phi' ( l p )' \, dx  =  \int_{\mathbb{R}} \phi (l p )'' \, dx.
    \end{equation}
    Using this technique, we have
    \begin{align}
    \mathbb{E}[ \mathcal{A}_t\phi(X_t) |\mathcal{F}^{Y}_t] &= \int_{\mathbb{R}} \phi \Big[ \underbrace{  \gamma p - (u p)' - (k \varepsilon p)' - (v \zeta p)' + \tfrac{1}{2} ( k^2 p + v^2 p)'' }_{\mathcal{A}_t^{\dag}p(x,t)} \Big]\, dx, \\
        \mathbb{E}[ \mathcal{B}_t\phi(X_t) |\mathcal{F}^{Y}_t] &= \int_{\mathbb{R}} \phi \Big[ \underbrace{ \varepsilon p - (k p)'}_{\mathcal{B}_t^{\dag}p(x,t)} \Big] \, dx.
    \end{align}
    Plugging $\mathcal{A}_t^{\dag}p(x,t)$ and $\mathcal{B}_t^{\dag}p(x,t)$ into \cref{eq:formal_PDE_dt}--\cref{eq:formal_PDE_dY} completes the proof.
\end{proof}

\subsection{Derivation of the class of particle filters}\label{prf:class}
\begin{proof}[Proof of \cref{prop:class}]
	Equation \cref{eq:PDE_dt} is a second-order ODE, consisting of three parts.
	A trivial solution to this equation can be obtained by setting each part to zero.
	Alternatively, one can use \cref{eq:PDE_dY} to first change the form of \cref{eq:PDE_dt} and then set each part to zero.
	Specifically, the second-order term $\smash{\tfrac{\partial^2}{\partial x^2} ( k^2 p)}$ in \cref{eq:PDE_dt} can be converted from the second-order derivative into the first-order derivative using \cref{eq:PDE_dY} as follows:
	\begin{equation}\label{eq:class_proof_1}
    	\tfrac{\partial^2}{\partial x^2} ( k^2 p) = \tfrac{\partial}{\partial x} \big(  k k' p  + k ( k p)' \big)  = \tfrac{\partial}{\partial x} \big( k ( k'  - h + \hat{h}_t + \varepsilon )  p \big).
	\end{equation}  
	Likewise, the zero-order term  $(h - \hat{h}_t)\hat{h}_t p$ in \cref{eq:PDE_dt} can yield a first-order derivative using \cref{eq:PDE_dY}:
	\begin{equation} \label{eq:class_proof_2}
	(h - \hat{h}_t )\hat{h}_t p = - \tfrac{\partial}{\partial x} ( k \hat{h}_t p ) +  \varepsilon \hat{h}_t p.
	\end{equation}
	To obtain a general expression, we may simply split
	\begin{align}
		\tfrac{\partial^2}{\partial x^2} ( k^2 p) & = \beta \tfrac{\partial^2}{\partial x^2} ( k^2 p)+(1-\beta)\tfrac{\partial^2}{\partial x^2} ( k^2 p),
		\\
		(h - \hat{h}_t ) \hat{h}_t p & = \alpha (h - \hat{h}_t ) \hat{h}_t p + (1-\alpha) (h - \hat{h}_t ) \hat{h}_t p,
	\end{align}
	where $\alpha,\beta \in \mathbb{R}$ are free parameters, which can be (continuous) functions of $t$, but not of $x$, because we would like to have the possibility to take them inside the derivatives in the next calculation steps.
	In the equations above, we let the first parts remain unchanged, while for the second parts, we use \cref{eq:class_proof_1} and \cref{eq:class_proof_2}.
	Equation \cref{eq:PDE_dt} then becomes
	\begin{multline}\label{eq:PDE_dt_alpha-beta}
	\smash{\frac{1}{2} \frac{\partial^2}{\partial x^2} \big( ( \beta k^2 + v^2 - g^2 ) p \big)  }\\
	 - \frac{\partial}{\partial x} \Big( \big( u - f + \frac{1}{2} k \big( (1-\beta) h + (1+\beta-2\alpha) \hat{h}_t + (1+\beta )\varepsilon -(1-\beta) k' \big) + v \zeta \big)  p \Big)  \\ \smash{  + \big( \alpha (h - \hat{h}_t)\hat{h}_t+ (1-\alpha) \varepsilon  \hat{h}_t + \gamma \big) p = 0.}
	\end{multline}
	In the equation above, by separately setting each of the three terms to zero, we obtain a particular solution.
	Furthermore, we wish to derive the results for a $\varepsilon$ that consists of an interpolation.
	As we discussed in \cref{subsec:class}, the choice of $\varepsilon$ that linearly interpolates between $h-\hat{h}_t$ and $0$ is a relevant choice for our purpose. We also let $\varepsilon$ consist of a free function (to not restrict ourselves to just an interpolation). So we take $\varepsilon = \eta (h-\hat{h}_t ) + \tilde{\varepsilon}$, where $\eta$ is a free parameter like $\alpha,\beta$, and $\tilde{\varepsilon}$ is an arbitrary function with zero mean under the posterior distribution (to guarantee the condition \cref{eq:zero-mean}).
	To sum up, plugging in $\varepsilon$ to the modified ODE \cref{eq:PDE_dt_alpha-beta} and setting each term to zero yield the results of this proposition.
\end{proof}

\subsection{Derivation of the particle distribution}\label{prf:q}
Here we first present a lemma that describes the evolution equation of $q(x,t)$ for a particle dynamics of the general form \cref{eq:HPF_dS}.
We then apply the result to the linear-Gaussian case in particular. This lemma will also be used later for the proof of \cref{prop:dW_optimal}.
\begin{lemma}[Proposition 1 in \cite{FPF2016}]\label{lemma:particle_distribution}
	Consider the stochastic process $S_t$ satisfying the SDE \cref{eq:HPF_dS} under the assumptions \ref{asmp:BM}--\ref{asmp:functions_S}. 
	The probability density function of $S_t$ given $\mathcal{F}^Y_t$, denoted by $q(x,t)$, evolves as the following SPDE:
	\begin{equation}\label{eq:dq}
	    dq(x,t) = \mathcal{J}^{\dagger}_t q(x,t)\,dt + \mathcal{K}_t^{\dagger} q(x,t) \, dY_t ,
	\end{equation}
    where $\mathcal{J}_t^{\dagger}\cdot := - \frac{\partial}{\partial x}(u(x,t)\cdot) + 
    \frac{1}{2} \frac{\partial^2}{\partial x^2} (k^2(x,t) \cdot+v^2(x,t) \cdot)$ and $\mathcal{K}_t^{\dagger} \cdot := - \frac{\partial}{\partial x} (k(x,t) \cdot )$.
\end{lemma}

\begin{proof}[Proof of statement \cref{eq:particle_mean}--\cref{eq:particle_var}]
    After setting $\tilde{\varepsilon}=0$ and $\zeta=0$ in \cref{prop:class} for the linear-Gaussian setting, the coefficients of the particle dynamics are given by
        \begin{align}
            k(t) & = (1-\eta) c \hat{\rho}_t, \label{eq:subclass_k_linear}\\
    	    v^2(t) & = b^2 - \beta (1-\eta)^2 c^2 \hat{\rho}_t^2, \label{eq:subclass_v_linear} \\
    	    u(x,t) & =  ax - \tfrac{1}{2} (1-\eta) c^2 \hat{\rho}_t (\vartheta_1 x + \vartheta_2  \hat{\mu}_t ), \label{eq:subclass_u_linear}
    	\end{align}
    	where \cref{eq:subclass_k_linear} is obtained by solving \cref{eq:PDE_dY}.
    	Due to the linearity and initial Gaussian distribution, the particle distribution will stay Gaussian, $S_t|\mathcal{F}^Y_t \sim \mathcal{N}(\mu_t,\rho_t)$. 
    	Hence, it only remains to find the evolution of the mean $\mu_t = \mathbb{E}[S_t | \mathcal{F}^{Y}_t] $ and variance $\rho_t = \text{Var}[S_t | \mathcal{F}^{Y}_t]$.
    	Using integration by parts, we can derive from \cref{eq:dq} that
    	\begin{align}
    	d\mathbb{E}[S_t | \mathcal{F}^{Y}_t] & = 
    	\mathbb{E} [u(S_t,t) | \mathcal{F}^{Y}_t]\,dt +
    	\mathbb{E} [ k (S_t,t) | \mathcal{F}^{Y}_t ] \, dY_t, \label{eq:dE[S|F]}\\
    	d\text{Var}[S_t | \mathcal{F}^{Y}_t]  & = 
    	\mathbb{E} [(S_t- \mathbb{E}[S_t | \mathcal{F}^{Y}_t])u(S_t,t) + v^2(S_t,t) | \mathcal{F}^{Y}_t]\,dt \label{eq:dVar[S|F]} \\
    	& + \text{Var} [k(S_t,t) | \mathcal{F}^{Y}_t]\,dt 
    	+  \mathbb{E} [(S_t- \mathbb{E}[S_t | \mathcal{F}^{Y}_t])k(S_t,t) | \mathcal{F}^{Y}_t]\,dY_t. \nonumber
    	\end{align}
    	Substituting \cref{eq:subclass_k_linear}--\cref{eq:subclass_u_linear} into \cref{eq:dE[S|F]}--\cref{eq:dVar[S|F]} yields the result.
\end{proof}

\subsection{Derivation of the optimal weight dynamics}\label{prf:dW_optimal}
\begin{proof}[Proof of \cref{prop:dW_optimal}]
    The goal is to obtain the stochastic differential $\smash{d(\tfrac{p(S_t,t)}{q(S_t,t)})}$, for which we now know that $p(x,t)$ and $q(x,t)$ satisfy the SPDEs \cref{eq:KSE_dis} and \cref{eq:dq}, respectively, and $S_t$ solves the SDE \cref{eq:HPF_dS}. 
    As a first step, we calculate by Itô's lemma
    \begin{equation}
        d\big(\tfrac{p(x,t)}{q(x,t)}\big) = \Lambda_1 (x,t) \, dt + \Lambda_2 (x,t) \, dY_t,
    \end{equation}
    where the functions $\{ \Lambda_1 (x,t), \Lambda_2 (x,t) \}$ are given by
    \begin{align}
        \Lambda_1 & := \tfrac{p}{q} \big( \tfrac{1}{p} \mathcal{L}_t^{\dagger}p - (h - \hat{h}_t) \hat{h}_t - \tfrac{1}{q} \mathcal{J}^{\dagger}_t q - \tfrac{1}{q} (h - \hat{h}_t) \mathcal{K}_t^{\dagger} q  + \tfrac{1}{q^2} (\mathcal{K}_t^{\dagger} q )^2 \big) , \\
        \Lambda_2 & := \tfrac{p}{q} \big( (h - \hat{h}_t ) - \tfrac{1}{q} \mathcal{K}_t^{\dagger} q \big).
    \end{align}
    Next, to calculate $\smash{d(\tfrac{p(S_t,t)}{q(S_t,t)})}$, the Kunita--Itô--Wentzell formula can be used (see \cite[Theorem 1.1]{Kunita1981}, and set $F_t(x) = \tfrac{p(x,t)}{q(x,t)}$, $M_t = S_t$ as continuous semimartingales). We get
    \begin{equation}
        d\big(\tfrac{p(S_t,t)}{q(S_t,t)}\big) = \Upsilon_1 (S_t,t) \, dt + \Upsilon_2 (S_t,t) \, dY_t + \Upsilon_3 (S_t,t) \, dB_t,
    \end{equation}
    where the functions $\{ \Upsilon_1 (x,t), \Upsilon_2 (x,t), \Upsilon_3 (x,t) \}$ are given by
    \begin{equation} \label{eq:Upsilon}
        \Upsilon_1 := \Lambda_1 + \mathcal{J}_t \tfrac{p}{q} +  \mathcal{K}_t \Lambda_2, \quad
        \Upsilon_2 := \Lambda_2 + \mathcal{K}_t \tfrac{p}{q}, \quad
        \Upsilon_3 := \mathcal{V}_t \tfrac{p}{q},
    \end{equation}
    with operators $\mathcal{J}_t \cdot := u(x,t) \frac{\partial}{\partial x} \cdot +  \frac{1}{2} (k^2(x,t) +v^2(x,t) ) \frac{\partial^2}{\partial x^2} \cdot $ and  $\mathcal{K}_t \cdot := k(x,t) \frac{\partial}{\partial x} \cdot $ whose adjoints have been already introduced in \cref{lemma:particle_distribution}.
    We have also defined the operator $\mathcal{V}_t \cdot := v(x,t) \frac{\partial}{\partial x} \cdot $.
    It is easy to see that $\Upsilon_3 = \tfrac{p}{q} (  v (\log \tfrac{p}{q})') = : \tfrac{p}{q} \zeta^*$, which proves \cref{eq:zeta^*}.
    Further, it can be verified that $\Upsilon_2$ and $\Upsilon_1$ correspond to $\tfrac{p}{q} \varepsilon^*$ and $\tfrac{p}{q} \gamma^*$, respectively, where $\varepsilon^*$ and $\gamma^*$ are $P_t$-a.s. unique solutions to \cref{eq:PDE_dY}--\cref{eq:PDE_dt} after setting $\zeta= \zeta^*$.
    The calculation steps are, however, omitted on account of space.
    Finally, to show that $\varepsilon^*,\gamma^*$ also satisfy \cref{eq:zero-mean}, it is easy to see that the optimal weight satisfies $\mathbb{E}[W_t^*|\mathcal{F}^{Y}_t]=1$ a.s., which leads to the same result as shown in \cref{eq:zero-mean_E}.
\end{proof}

\subsection{Derivation of the Fisher-optimal constant gain approximation}\label{prf:gain}
\begin{proof}[Proof of \cref{prop:optimal_constant_gain}]
    Given the definition of $\psi(x,t)$, the objective functional $\mathcal{I}[{\bar{k}}]$ can be written as
	\begin{multline}
	    \mathbb{E}[ {\varepsilon}^2(X_t,t) | \mathcal{F}^{Y}_t] = \mathbb{E}[ (h(X_t,t) - \hat{h}_t)^2 | \mathcal{F}^{Y}_t] \\ +  \bar{k}(t)^2 \mathbb{E}[ \psi^2(X_t,t) | \mathcal{F}^{Y}_t]  
	    - 2 \bar{k}(t) \mathbb{E}[ h'(X_t,t) | \mathcal{F}^{Y}_t] \quad a.s.
	\end{multline}
	where we used integration by parts for the last term and the boundedness of $h$.
	For a fixed $t$ and $p$, we take the derivative of the objective functional with respect to $\bar{k}(t)$:
	\begin{equation}\label{eq:1st_derivative_functional}
	\frac{d}{d\bar{k}}\mathbb{E}[ {\varepsilon}^2(X_t,t) | \mathcal{F}^{Y}_t] = 2 \bar{k}(t) \mathbb{E}[ \psi^2(X_t,t) | \mathcal{F}^{Y}_t] - 2 \mathbb{E}[ h'(X_t,t) | \mathcal{F}^{Y}_t] \quad a.s.
	\end{equation}
	Furthermore, observe that the second derivative is a.s. positive:
	\begin{align}
	\frac{d^2}{d\bar{k}^2}\mathbb{E}[ {\varepsilon}^2(X_t,t) | \mathcal{F}^{Y}_t] = 2 \mathbb{E}[ \psi^2(X_t,t) | \mathcal{F}^{Y}_t] > 0 \quad a.s.
	\end{align}
	Thus, the minimizer $\bar{k}^*(t)$ is found by setting the first derivative \cref{eq:1st_derivative_functional} to zero.
\end{proof}

\subsection{Derivation of weight degeneracy dynamics}\label{prf:dEW2}
\begin{lemma}[Proposition 2.30. in \cite{Bain2009}]\label{lemma:innovation_process}
	Consider the filtering problem \cref{eq:hidden}--\cref{eq:observation}. The innovation process $I_t$, which evolves according to the SDE $ dI_t = dY_t - \hat{h}_t dt$, is an $\mathcal{F}_t^Y$-adapted BM if $h(X_{\cdot},\cdot) \in \mathbb{L}^1(0,t)$ holds.
\end{lemma}

\begin{proof}[Proof of \cref{prop:weight_degeneracy}]
    Plug $dY_t = dI_t + \hat{h}_t dt$ into SDEs \cref{eq:dW^2}--\cref{eq:dlogW}, write them in integral form, and take the expectation of both sides.
    In addition to $B_t$, the innovation process $I_t$ is a BM by \cref{lemma:innovation_process}.
    Hence, terms that involve $B_t$ and $I_t$ vanish due to the fact that the expectation of Itô's Integral is zero.
    Finally, applying Fubini’s theorem (see, e.g., \cite[Theorem 2.39]{Klebaner2005}) for the remaining terms and turning the result back into the differential form give the results. 
\end{proof}

\vspace{1pt}

\bibliographystyle{siamplain}
\bibliography{ms}

\end{document}